\documentclass[12pt]{amsart}
\usepackage{graphicx}
\usepackage{latexsym}
\usepackage{amscd}
\usepackage{epsfig}
\usepackage{float}
\usepackage{amsmath, amssymb, amsthm, amsfonts, amsgen} 
\usepackage{stmaryrd}
\usepackage{todonotes}
\usepackage{epstopdf}
\usepackage{xcolor}
\usepackage{comment}
\usepackage{url}
\usepackage{hyperref}
\usepackage{float}
\usepackage{mathtools}
\usetikzlibrary{decorations.pathreplacing}
\usetikzlibrary{decorations.pathreplacing,shapes,shadows,arrows,positioning,graphs,calc,patterns}

\graphicspath{ {./} }

\numberwithin{equation}{section}
\newtheorem{definition}{Definition}[section]
\newtheorem{lemma}[definition]{Lemma}
\newtheorem{theorem}[definition]{Theorem}
\newtheorem{proposition}[definition]{Proposition}

\newtheorem{remark}[definition]{Remark}

\newcommand{\R}{\mathbb{R}}
\newcommand{\N}{\mathbb{N}}

\newcommand{\om}{\Omega}
\newcommand{\Div}{{\rm div}\,}

\newcommand{\cof}{{\rm cof}\,}

\newcommand{\1}{{\bf 1}}

\newcommand{\eps}{\epsilon}

\newcommand{\sca}{\mathcal{A}}
\newcommand{\sco}{a}     

\newcommand{\sch}{\mathcal{H}}

\newcommand{\D}{\mathbb{D}}

\newcommand{\md}{{\rm d}}
\newcommand{\dx}{\, \md x}
\newcommand{\dy}{\, \md y}

\newcommand{\dHone}{\, \md \mathcal{H}^1}

\newcommand\EEE{\color{black}}
\newcommand{\mk}{\color{blue}}

\def\Xint#1{\mathchoice
   {\XXint\displaystyle\textstyle{#1}}%
   {\XXint\textstyle\scriptstyle{#1}}%
   {\XXint\scriptstyle\scriptscriptstyle{#1}}%
   {\XXint\scriptscriptstyle\scriptscriptstyle{#1}}%
   \!\int}
\def\XXint#1#2#3{{\setbox0=\hbox{$#1{#2#3}{\int}$}
     \vcenter{\hbox{$#2#3$}}\kern-.5\wd0}}

\def\dashint{\Xint-}

\begin{document}
\author[J. Bevan]{Jonathan J. Bevan}
\address[J. Bevan]{Department of Mathematics, University of Surrey, Guildford, GU2 7XH, United Kingdom.  (Corresponding author: \textbf{t:} $+44 (0)1483 \  682620$.)}
\email[Corresponding author]{j.bevan@surrey.ac.uk}

\author[M. Kru\v{z}\'{i}k]
{Martin Kru\v{z}\'{i}k}
\address[M. Kru\v{z}\'{i}k]{ Czech Academy of Sciences, Institute of Information Theory and Automation,
Pod vod\'arenskou v\v{e}\v z\'\i\ 4, 182 00, Prague 8, Czechia $\&$
Department of Physics, Faculty of Civil Engineering, Czech Technical University in Prague, Thákurova 7, 166 29 Prague 6, Czechia.
}
\email{kruzik@utia.cas.cz}

\author[J. Valdman] {Jan Valdman} 
\address[J. Valdman]{Czech Academy of Sciences, Institute of Information Theory and Automation, Pod vod\'arenskou v\v{e}\v z\'\i\ 4, 182 00, Prague 8, Czechia $\&$
Department of Mathematics, Faculty of Information Technology, Czech Technical University in Prague, Thákurova 9, 16000 Prague 6, Czechia.
}
\email{jan.valdman@utia.cas.cz}


\title[Applications of HIM]{New applications of Hadamard-in-the-mean inequalities to incompressible variational problems}
\begin{abstract} Let $\D(u)$ be the Dirichlet energy of a map $u$ belonging to the Sobolev space $H^1_{u_0}(\om;\R^2)$ and let $\sca$ be a subclass of $H^1_{u_0}(\om;\R^2)$ whose members are subject to the constraint $\det \nabla u = g$ a.e.\@ for a given $g$, together with some boundary data $u_0$.  We develop a technique that, when applicable, enables us to characterize the global minimizer of $\D(u)$ in $\sca$ as the unique global minimizer of the associated functional $F(u):=\D(u)+ \int_{\om} f(x) \, \det \nabla u(x) \, \dx$ in the free class $H^1_{u_0}(\om;\R^2)$. A key ingredient is the mean coercivity of $F(\varphi)$ on $H^1_0(\om;\R^2)$, which condition holds provided the `pressure' $f \in L^{\infty}(\om)$ is `tuned' according to the procedure set out in \cite{BKV23}.  The explicit examples to which our technique applies can be interpreted as solving the sort of constrained minimization problem 
that typically arises in incompressible nonlinear elasticity theory.
\end{abstract}
\maketitle

\section{Introduction}

The chief purpose of this paper is to give a new viewpoint on the classical problem of finding global minimizers of constrained variational problems that are typically encountered in incompressible nonlinear elasticity theory.  The novelty of our technique is that in the cases where it applies, it delivers a unique global energy minimizer in a constrained class by first solving an explicit PDE that is set in an unconstrained (or free) class. Established techniques for treating such variational problems can, in the right circumstances, produce similar PDE as necessary conditions, but without an associated uniqueness principle that makes it possible to distinguish between stationary points and true minimizers.


To be more specific, our approach puts to use the recent results in \cite{BKV23} treating the mean coercivity of functionals of the form 
\begin{align}\label{Fintro}
    F(\varphi):= \int_{\om} |\nabla \varphi|^2 + f(x) \det \nabla \varphi \dx
\end{align}
defined for $\varphi$ in $H_0^1(\om;\R^2)$ and where $f$ is given function in $L^{\infty}(\om)$.  When $F$ is mean coercive on $H_0^1(\om;\R^2)$, i.e. when there is a constant $\gamma>0$ such that 
\begin{align}
    \label{MMC} F(\varphi) \geq \gamma \int_{\om} |\nabla \varphi|^2 \, \dx \quad \forall \varphi \in H_0^1(\om;\R^2),
\end{align}
it becomes possible to minimize $F$ over the general class 
$$H^1_{u_0}(\om;\R^2):=\{
u \in H^1(\om;\R^2): \ u=u_0 \ \textrm{ on } \ \partial \om
\}$$
and so obtain a unique solution $u$ to the associated Euler-Lagrange equation
\begin{align}\label{ELintro}
    \int_{\om} (2 \nabla u + f(x) \, \cof \nabla u) \cdot \nabla \psi = 0 \quad \forall \psi \in H_0^1(\om;\R^2).
\end{align}
This is precisely the type of equation we would expect to have to solve in order to minimize
$$\D(u):=\int_{\om}|\nabla u|^2 \, \dx$$
in the constrained class 
\begin{align}\label{ag} \sca_g:=\{u \in H^1_{u_0}(\om;\R^2): \  \det \nabla u = g \ a.e.\}\end{align}
provided the given function $g$ is such that $\sca_g$ is nonempty.  In this case, the function $f$ is a Lagrange multiplier corresponding to the constraint $\det\nabla u=g$ a.e., while in the continuum mechanics literature $f$ is interpreted as a (hydrostatic) pressure. 

Indeed, in the general case of the constrained variational problem of minimizing the stored energy functional 
\begin{align}\label{E1}
E(u):=\int_{\om}W(\nabla u) \, \dx 
\end{align}
in a class typified by $\sca_g$, the established approach is first to determine by the Direct Method of the Calculus of Variations that a minimizer of $E$ in $\sca_g$ exists and then to derive a version of \eqref{ELintro}, namely
\begin{align}\label{ELgc}
\int_{\om} (DW(\nabla u) + \lambda(x) \, \cof \nabla u) \cdot \nabla \psi = 0 \quad \forall \psi \in C_0^\infty(\om,\R^n)
\end{align}
for a suitable $\lambda$ and where $n$ is typically $2$ or $3$.  One of the clearest expositions of this technique can be found in \cite{FRC99}, where minimizers of an energy like \eqref{E1} in a subclass of the isochoric maps $u$ obeying $\det \nabla u = 1$ a.e.\@ are shown to obey   \eqref{ELgc}.  The heart of the argument is to assume that the global minimizer $y^*$ of $E$ in the constrained class is sufficiently regular and invertible that the problem can be formulated in the deformed configuration $y^*(\om)$.  See also \cite{LO81,FM86}.  In the same spirit, and again in subclasses of isochoric maps, equations of the type \eqref{ELgc} are derived in \cite[Section 4]{Chkara09} for continuous and injective local minimizers satisfying suitable regularity conditions which include the assumption that $(\nabla u)^{-T} \in L_{\textrm{loc}}^s(\om;\R^{n\times n})$ for $s\geq 3$ and $n=2,3$. The resulting hydrostatic pressure $\lambda$ then belongs to $L_{\textrm{loc}}^{\frac{s}{2}}(\om)$. See also \cite{Kara12,Kara14} for further results in this direction.

By contrast with the classical approach described above, which fixes the constraint $\det \nabla u = g$ a.e.\@ in advance (e.g. by setting $g\equiv 1$ in the isochoric case), our process is as follows:
\begin{enumerate}
    \item[1.] Relying on the results of \cite{BKV23}, choose a pressure $f$ in \eqref{Fintro} so that the functional $F$ is mean coercive in the sense of \eqref{MMC}, and take $u_0 \in H^1(\om;\R^2)$;
    \item[2.] Minimize $F$ in the `free' class $H^1_{u_0}(\om;\R^2)$ and let $u$ be a minimizer.  Note that $u$ solves the weak Euler-Lagrange equation $a(u,\psi)=0$ for all test functions $\psi$, where the Euler-Lagrange operator is defined in \eqref{ELop} below;
\item[3.] Define $g:=\det \nabla u$ and let the constrained class of admissible maps $\sca_g$ be given by \eqref{ag};
\item[4.] Employ the identity 
\begin{align}\label{IDENTITY}
    F(v) = F(u) + a(u,v-u) + F(v-u) \quad \quad u,v \in H_{u_0}^1(\om;\R^n),
\end{align}
which, because $u$ solves the Euler-Lagrange equations and by imposing the conditions that $\det \nabla u = \det\nabla v = g$ a.e.\@, simplifies to 
\begin{align}\label{pyotr}
 \mathbb{D}(v) =  \mathbb{D}(u) + F(v-u) \quad \quad v \in \sca_g,
\end{align}
and note that by mean coercivity, $F(v-u)> 0$ for any $v,u \in \sca_g$ such that $v \neq u$.  It follows that $u$ is the unique global minimizer of the Dirichlet energy in $\sca_g$. 
\end{enumerate} 

See Proposition \ref{basicfact} for \eqref{IDENTITY} and Lemma \ref{diddly} for \eqref{pyotr}.
In practice, the choice of the pressure $f$ dictates the possible solutions of the Euler-Lagrange equation which, in turn, control the permissible boundary conditions $u_0\arrowvert_{\partial \om}$.  We study in Sections \ref{disk-disk}, \ref{disk-sector}, \ref{insulation_f} and \ref{sudoku_f} solutions of the Euler-Lagrange equation for four main types of pressure function $f$, and calculate in each case global energy minimizers of classes of constrained minimization problems as formulated in Steps 1.-4.\@ above.  In Section \ref{disk-disk}, for example, we prove that for suitable constants $\zeta$ and $\xi$ the map 
\begin{align*}
    u(x):=\left\{\begin{array}{l l} \zeta x & \ \ x \in B(0,\rho) \\
    \left(\xi+\frac{1-\xi}{|x|^2}\right)x & \ \ x \in B(0,1)\setminus B(0,\rho) \end{array}\right.
    \end{align*}
is the unique global minimizer of the Dirichlet energy in the class 
\begin{align*}
    \{v \in H^1_{\textrm{id}}(B(0,1);\R^2): \ \det \nabla v = g \ \textrm{a.e.}\}
\end{align*}
where $g(x):=\zeta^2$ if $x\in B(0,\rho)$ and $g(x):=\xi^2-(1-\xi)^2 |x|^{-4}$ otherwise.  Here, $B(0,\rho)$ stands as usual for the ball in $\R^2$ centered at $0$ and of radius $\rho$. 

In this and the other examples mentioned above, the maps we work with are planar, i.e. they take $\om \subset \R^2$ into $\R^2$.  We are also able to extend our analysis to the functional 
 \begin{align}\label{island3d}
     \tilde{F}(u):=\int_\om |\nabla u|^2 + T \cdot \cof \nabla u \dx,
 \end{align}
 where $u: \om \subset \R^3 \to \R^3$ and $T \in L^{\infty}(\om,\R^{3 \times 3})$ is a given matrix-valued function which, as is explained in Section \ref{3dislandexample}, causes $\tilde{F}$ to be mean coercive provided $||T||_{\infty} < 2\sqrt{3}$.  $\tilde{F}$ then has a unique global minimizer in the class 
 \begin{align}\label{a2}
    \sca_2:=\{u \in H^1(\om;\R^3): u\arrowvert_{\partial \om}=u_0\},
\end{align}
where $u_0$ is the trace of a fixed function in $H^1(\om;\R^3)$, and a process analogous to that outlined in Steps 1.-4.\@ can be followed.

One of the features of the functional $F$ in the $2\times 2$ case is that its integrand $W(x,A):=|A|^2 + f(x)\det A$ for $A \in \R^{2 \times 2}$
does not, for general $f$, satisfy a pointwise ellipticity condition of Legendre-Hadamard type
\begin{align}\label{LH} D^2W(x,A)[a \otimes b, a \otimes b] \geq \nu |a|^2|b|^2 \quad \quad A, a\otimes b \in \R^2, \ x \in \om.\end{align}
Using Hadamard's pointwise inequality $|A|^2 \geq 2 |\det A|$ for all $A \in \R^{2\times 2}$, it is straightforward to see that such a condition holds only if $||f-(f)_{\om}||_{\infty} < 2$, where $(f)_{_{\om}}$ denotes the mean value of $f$ over $\om$.  Nevertheless, even when \eqref{LH} fails it is still possible to show that the mean coercivity of $F$ is sufficient to improve the regularity of $H^1$ solutions of the associated Euler-Lagrange equation to $C^{0,\alpha}$ for some $\alpha >0$. See Proposition \ref{continuity} for details and its preamble for a discussion of this result in relation to those of Morrey \cite[Theorem 4.3.1]{Mo66} and Giaquinta and Giusti \cite{GiaquintaGiusti1982}.  Thus when, for example, we take $f(x)=M \chi_{_{\omega}}(x)$, where $\chi_{_{\omega}}$ is the characteristic function of the fixed subdomain $\omega \subset \om$ and where the scalar $M$ obeys $|M|<4$, the solution to the Euler-Lagrange equation \eqref{ELintro} is H\"{o}lder continuous, despite the evident discontinuity in $f$. In fact, in this and other such cases, the Euler-Lagrange equation splits into a `bulk part', leading to the conclusion that the solution $u$ is harmonic away from $\partial \omega$, and a surface part, where certain jump conditions relating the normal and tangential derivatives of $u$ along $\partial \omega$ should hold.  See Proposition \ref{howtosolveEL} for this interpretation and the assumptions we make to derive it.

The paper is organised as follows. In Section \ref{sectiontwo} the functional $F$ given by \eqref{Fintro} is studied under the assumption that it is mean coercive, and the properties of solutions to the associated Euler-Lagrange equations are derived, including Proposition \ref{continuity}, which guarantees the H\"{o}lder continuity mentioned above.  The important decomposition \eqref{IDENTITY} is derived in Proposition \ref{basicfact}, and a result that is the blueprint for solving the Euler-Lagrange equations appearing throughout the paper is established in Proposition \ref{howtosolveEL}.  Subsections \ref{disk-disk} and \ref{disk-sector} focus on two cases in which the pressure $f$ is of the form $f=M\chi_{\omega}$ and $\omega$ is either a disk or a sector.  Section \ref{section3} focusses on the constrained variational problems generated by taking $f$ to be of two further forms: see Section \ref{insulation_f} for a setting in which the global minimizer turns out to be piecewise affine, and Section \ref{sudoku_f} for a setting in which minimizers can be generated only if the parameters appearing in the pressure $f$ are carefully selected.  

We denote by $J$ the $2 \times 2$ matrix representing a rotation by $\pi/2$ radians anticlockwise, i.e. in terms of the canonical basis vectors $e_1$ and $e_2$ in $\R^2,$ $J = e_2 \otimes e_1 - e_1 \otimes e_2$.  Other than that, all notation is either standard or else is defined when first used. 

\section{Minimizing the functional $F$ under mean coercivity conditions}\label{sectiontwo}

The subsection title refers to the variational problem of minimizing the energy $F$ defined by
\begin{align}\label{Finhom}
    F(u):=\int_{\om} |\nabla u|^2 +f(x) \det \nabla u\dx
\end{align}
in the class of admissible maps 
\begin{align}\label{firstadmiss}
    H_{u_0}^1(\om;\R^2)=\{u \in H^1(\om;\R^2): u\arrowvert_{\partial \om}=u_0\},
\end{align}
where $u_0$ is the trace of a fixed function in $H^1(\om;\R^2)$.  Here, $f$ is a fixed function in $L^{\infty}(\om)$, which we may sometimes refer to as a `pressure', chosen so that $F$ is mean coercive, by which we mean that there is $\gamma>0$ such that 
\begin{align}\label{initialMC} F(\varphi) \geq \gamma \int_{\om} |\nabla \varphi|^2 \, \dx \quad \forall \varphi \in H_0^1(\om;\R^2).
\end{align}
Conditions on $f$ ensuring that \eqref{initialMC} holds can be found in \cite{BKV23}, to which point we will return later. By a straightforward density argument, we remark that the space $H_0^1(\om;\R^2)$ appearing in \eqref{initialMC} can be replaced with the set of smooth, compactly supported test functions on $\om$.  

The connection between mean coercivity and the existence of minimizers of $F$ is recorded in the following result. 

\begin{proposition}\label{basicfact} Let $u,v \in H_{u_0}^1(\om;\R^2)$ and let $F$ be given by \eqref{Finhom}.  Then 
\begin{align}\label{MKdecomposition}
F(v)=F(u) + \sco(u,v-u)+F(v-u) 
\end{align}
where $\sco(u,\varphi)$ represents the bilinear operator 
\begin{align}\label{ELop}
  \sco(u,\varphi):=\int_\om 2\nabla u \cdot \nabla\varphi+f(x)\, \cof\nabla u\cdot \nabla\varphi \dx.
\end{align}
If $F$ is mean coercive then it has a unique minimizer $u \in H_{u_0}^1(\om;\R^2)$ obeying the Euler-Lagrange equation 
 \begin{align}\label{ELinhombc}
  \sco(u,\varphi) = 0  \quad \quad  \forall \varphi \in H_0^1(\om;\R^2).
  \end{align}
\end{proposition}
\begin{proof} Writing $v=u+\varphi$, expanding the determinant 
$$\det (\nabla u+\nabla \varphi) = \det \nabla u + \cof \nabla u \cdot \nabla \varphi + \det \nabla \varphi$$
and substituting in $F(u+\varphi)$ yields the decomposition \eqref{MKdecomposition}.  When $F$ is mean coercive, the direct method of the Calculus of Variations yields a minimizer $u$, say, in $H_{u_0}^1(\om;\R^2)$, and by taking suitable variations, it must be that $u$ obeys \eqref{ELop}.  The uniqueness follows by applying \eqref{MKdecomposition} and \eqref{initialMC} to deduce that for any other candidate minimizer $v$, say, 
\begin{align*}
    F(v) & \geq F(u) + \gamma \int_{\om} |\nabla \varphi|^2 \, dx \end{align*}
and, by exchanging $u$ and $v$,    
\begin{align*} 
    F(u) & \geq F(v) + \gamma \int_{\om} |\nabla \varphi|^2 \, dx. 
\end{align*}
These are consistent only if $\varphi=0$ a.e., which gives $v=u$ a.e..
\end{proof}

We remark that the decomposition \eqref{MKdecomposition}
shows that if there is just one test function $\varphi$ such that $F(\varphi)<0$
then $F(u+k\varphi) \to -\infty$ as $k \to \infty$, and there is no infimum, let alone a minimizer. Hence if there is a finite infimum, it is necessary that 
\begin{align}\label{bepositive} F(\varphi) \geq 0 \quad \forall \varphi \in H_0^1(\om;\R^2).\end{align}
Mean coercivity is therefore a natural strengthening of this necessary condition.  Moreover, since it follows easily from \eqref{bepositive} that 
$$ \min\{F(\varphi): \  \varphi \in H_0^1(\om;\R^2)\}=0,$$
we deduce that if $F$ is in addition mean coercive then the unique minimizer of $F$ on $H_0^1(\om;\R^2)$ is $u=0$. We refer to \cite{CKKK} for other applications of convex 
integral functionals defined by possibly nonconvex integrands.

We now study the Euler-Lagrange equation \eqref{ELinhombc} for general $f$ in $L^{\infty}(\om)$ under the assumption that $f$ can be chosen so that $F$ is mean coercive, i.e. that \eqref{initialMC} holds.   This is a weaker assumption than ellipticity, as can be seen by considering the particular example of $f=M \chi_{\omega}$ where $\omega \subset \om$: 
 the system \eqref{ELinhombc} is elliptic only when $|M|<2$, whereas, by \cite[Proposition 3.4]{BKV23} it is mean coercive only when $|M|<4$.  Fortunately, classical regularity theory is readily adapted in order to exploit the mean coercivity condition \eqref{initialMC}.  Indeed, the conclusion of Proposition \ref{continuity} below echoes that of Giaquinta and Giusti \cite{GiaquintaGiusti1982}, in which an improvement in the regularity of a minimizer of certain nondifferentiable functionals is shown to be possible, and also that of the well-known result of Morrey \cite[Theorem 4.3.1]{Mo66}, but, in our case, without any pointwise growth assumptions on the integrand.  Specifically, we show that weak solutions to the Euler-Lagrange equation belong to the space $W^{1,p}_{\textrm{loc}}(\om;\R^2)$ for some $p>2$, and hence are, by Sobolev embedding, automatically locally H\"{o}lder continuous in $\Omega$. In the following,  we use the notation $(u)_{_S}:=\dashint_{S}u(y) \dy$ whenever $S \subset \om$ is measurable and non-null. 

\begin{proposition}\label{continuity} Let $u \in H_{u_0}^1(\om;\R^2)$ be a weak solution of the Euler-Lagrange equation
\begin{align}\label{ELpartic}
    \int_{\om}2\nabla u \cdot \nabla \varphi + f\, \cof\nabla u\cdot \nabla\varphi \dx = 0 \quad \quad \varphi \in W^{1,2}_0(\om,\R^2)
\end{align}
and assume that $F$ is mean coercive in the sense of \eqref{initialMC}.  Then there is $p>2$ such that $u$ belongs to $W^{1,p}_{\textrm{loc}}(\om;\R^2)$.
    \end{proposition}
\begin{proof} Let $x_0$ be any interior point of $\om$ and let $R_0>0$ be such that $B(x_0,2R)\subset \om$ for all $R \in (0,R_0)$.  Fix $R \in (0,R_0)$ and let $\eta$ be a smooth cut-off function with the properties that $\eta(x)=1$ for $x \in B(x_0,R)$, $\textrm{spt}\, \eta \subset B(x_0,2R)$ and $|\nabla \eta| \leq c/R$ for some constant $c$.  Let $\lambda$ be a constant vector in $\R^2$.  
Choosing $\varphi=\eta^2(u-\lambda)$ in \eqref{ELpartic} gives
\begin{align}\label{one}
    0 = &  \int_{\om} \eta^2 |\nabla u|^2 + \eta^2 f \,\det\nabla u \,\dx + \\ \nonumber & \qquad \quad \quad  + \int_{\om} 2 \eta (u-\lambda) \otimes \nabla \eta \cdot \nabla u + f  \eta (u-\lambda) \otimes \nabla \eta \cdot \cof \nabla u \, \dx. 
\end{align}
Now, \begin{align}\label{two}
    F(\eta(u-\lambda)) & = \int_{\om} \eta^2 |\nabla u|^2 + 2\eta(u-\lambda) \otimes \nabla \eta \cdot \nabla u + |u-\lambda|^2 |\nabla \eta|^2 \,\dx+ \\ \nonumber & \quad \quad + \int_{\om} \eta^2 f \, \det \nabla u + f \eta  (u-\lambda) \otimes \nabla \eta \cdot \cof \nabla u \, \dx,
\end{align}
which by applying \eqref{one} leads to 
\begin{align}\label{trill}
   F(\eta(u-\lambda)) = \int_{\om} |u-\lambda|^2 |\nabla \eta|^2 \dx. 
\end{align}
Since $\eta(u-\lambda)$ belongs to $W_0^{1,2}(\om,\R^2)$, we can apply \eqref{initialMC} to the left-hand side of the last equation, which gives, for some $\gamma>0$,
\begin{align*}
    \gamma\int_{\om} \eta^2 |\nabla u|^2 + 2 \eta (u - \lambda) \otimes \nabla \eta \cdot \nabla u + |u-\lambda|^2 |\nabla \eta|^2 \, \dx \leq\int_{\om} |u-\lambda|^2 |\nabla \eta|^2 \dx. 
\end{align*}
Hence there are constants $c_1,c_2,c_3$ and $\theta$ depending only on $\gamma$ and $f$ such that 
\begin{align*}
    \gamma \int_{B(x_0,R)}|\nabla u|^2 \, \dx & \leq c_1\int_{B(x_0,2R)\setminus B(x_0,R)} |u-\lambda|^2 |\nabla \eta|^2 \, \dx + \\ & \quad \quad\quad\quad \quad \quad\quad\quad  + c_2 \int_{B(x_0,2R)\setminus B(x_0,R)}|u-\lambda||\nabla \eta||\nabla u|\, \dx \\
    & \leq \frac{c_3}{R^2}\int_{B(x_0,2R)\setminus B(x_0,R)} |u-\lambda|^2 \, \dx + \theta \int_{B(x_0,2R)\setminus B(x_0,R)}|\nabla u|^2 \, \dx.
\end{align*}
where, without loss of generality, $4 \theta < \gamma$.  
Replacing the domain of integration on the right-hand side by $B(x_0,2R)$, dividing through by $\pi R^2$, taking $\lambda = (u)_{_{B(x_0,2R)}}$ and applying the Sobolev-Poincar\'{e} inequality in the form 
\begin{align*} \int_{B(x_0,2R)} |u-(u)_{_{B(x_0,2R)}}|^q \, \dx \leq C \left(\int_{B(x_0,2R)} |\nabla u|^{\frac{nq}{n+q}}   \right)^{\frac{n+q}{n}} \end{align*}
with $n=q=2$ leads eventually to 
\begin{align}\label{RVPC}
   \dashint_{B(x_0,R)} |\nabla u|^2 \, \dx & \leq \tilde{C} \left(\dashint_{B(x_0,2R)} |\nabla u|   \right)^2 + \frac{4 \theta}{\gamma} \dashint_{B(x_0,2R)} |\nabla u|^2 \, \dx. 
\end{align}
Since $\theta':=\frac{4 \theta}{\gamma}<1$, \eqref{RVPC} is a reverse H\"{o}lder inequality and, by applying  \cite[Proposition 1.1, Chapter V]{Giaquinta} with $q=2$ and $g=|\nabla u|$, we deduce that there is $\eps>0$ such that $\nabla u \in L^p_{\textrm{loc}}(\om)$ for any $p \in [2,2+\eps)$. It follows from this and Sobolev embedding that $u \in W^{1,p}_{\textrm{loc}}(\om,\R^2)$, as claimed. 
\end{proof}

A second interesting feature of the Euler-Lagrange equations \eqref{ELinhombc} is that, thanks to Proposition \ref{continuity} and  properties of null Lagrangians, the `cofactor part' of $\sco(u,\varphi)$ reduces to a `surface' integral when $f$ is a piecewise constant function and provided $u$ is regular enough.  We illustrate this initially by means of the following result by taking $f=M\chi_{\omega}$ in \eqref{Finhom}, and later, for different pressure functions $f$, in Propositions \ref{solveELcurrent} and \ref{nec2}.

\begin{remark}\emph{The problem of minimizing $F$ in $H_{u_0}^1(\om;\R^2)$ admits a physical interpretation in terms of the stored energy of a nonlinearly elastic material that is, in parts, subject to an applied dead-load pressure.  The associated PDE \eqref{ELinhombc} gives information both in the `bulk' (via harmonicity on $\om \setminus \partial \omega$) and on the `surface' $\partial \omega $ (via jump conditions.)  Furthering the connection with nonlinear elasticity, we may rewrite $F$ in terms of the Cauchy-Green stress tensor $C:=\nabla u^T \nabla u$ and note that in our case we have existence and uniqueness of equilibria under conditions that are not covered by the general results of \cite{Neffetal}.} \end{remark}

\begin{proposition}\label{howtosolveEL}
    Let the functional $F$ be given by \eqref{Finhom} with $f=M\chi_{\omega}$, and assume that $u \in H_{u_0}^1(\om;\R^2)$ solves the Euler-Lagrange equation \eqref{ELinhombc} for $F$.  Then
    \begin{itemize}\item[(i)] $u$ is harmonic in each of $\omega$ and $\om \setminus \omega$, and
    \item[(ii)] as long as these quantities exist
    \begin{align}\label{conslawboundary}2 \partial_\nu u\arrowvert_{\omega}+2 \partial_{-\nu} u\arrowvert_{\Omega\setminus \omega} - M J \partial_{\tau} u= 0 \quad \sch^1-a.e. \textrm{on} \ \partial \omega \setminus \partial \om, \end{align}
    \end{itemize}
     where the local normal $\nu$ and tangent $\tau$ are defined $\sch^1-$almost everywhere. 
\end{proposition}
\begin{proof}
     By a density argument, we may assume that $u$ solves $\sco(u,\varphi)=0$ for all $\varphi \in C_c^\infty(\om;\R^2)$.
     Using Piola's identity $ \Div\cof \nabla u=0$, we see that 
     \begin{align*}
         \int_\omega \cof \nabla u \cdot \nabla \varphi \dx = \int_{\partial \omega} \varphi \cdot \cof \nabla u \, \nu \dHone,
     \end{align*}
 and since $\cof A=J^T A J$ for any $2\times 2$ matrix $A$ and $J\nu =\tau$ in local coordinates on $\partial \omega$, we can write $\cof \nabla u \, \nu = - J \partial_\tau u$.  Hence the second term in $\sco(u,\varphi)$ obeys
 \begin{align}\label{surface}
     \int_\om f(x)\,\cof \nabla u \cdot \nabla \varphi \dx = - M \int_{\partial \omega} \varphi \cdot J \partial_\tau u \dHone, 
 \end{align}
 and the Euler-Lagrange equation reads
 \begin{align}
     \int_\om \nabla u \cdot \nabla \varphi \dx  - M \int_{\partial \omega} \varphi \cdot J \partial_\tau u \dHone = 0  \quad \quad  \varphi \in C_c^\infty(\om,\R^2).
     \end{align}
By choosing test functions $\varphi$ first with support only in $\omega$, and then with support only in $\om \setminus \omega$, the surface term involving $\partial_\tau u$ vanishes, and 
 it follows by  Weyl's lemma and standard theory that $u$ is harmonic in each of $\omega$ and its complement in $\om$.  Hence part (i) of the proposition.

 To prove (ii), use the harmonicity of $u$ in $\omega$ and then in $\om \setminus \omega$ to rewrite, for a general test function $\varphi$,
 \begin{align*}
     \int_\om \nabla u \cdot \nabla \varphi \dx & = \int_\omega\Div (\nabla u^T \varphi) \dx + \int_{\om \setminus \omega} \Div (\nabla u^T \varphi) \dx \\
     & = \int_{\partial \omega} \varphi \cdot \partial_\nu u\arrowvert_{\omega} \dHone + \int_{\partial \omega} \varphi \cdot \partial_{-\nu} u\arrowvert_{\Omega\setminus\omega} \dHone, 
     \end{align*}
     and combine with \eqref{surface} to obtain
     \begin{align*}
         \int_{\partial \omega} \left(2 \partial_\nu u\arrowvert_{\omega}+2 \partial_{-\nu} u\arrowvert_{\Omega\setminus \omega} - M J \partial_{\tau}u\right) \cdot \varphi  \dHone =0.
     \end{align*}
   Since $\varphi\arrowvert_{\partial \omega}$ is free other than on that part of $\partial \omega$ which meets $\partial \Omega$,  (ii) follows.   
 \end{proof}

In some special cases, using Proposition \ref{howtosolveEL} it is possible to solve the Euler-Lagrange equation \eqref{ELinhombc} explicitly.  

\subsection{The case that $\omega$ is a subdisk of the the unit ball in $\R^2$.} \label{disk-disk}
Let $\Omega=B(0,1)$ and, for a fixed $\rho  \in (0,1)$, let $\omega=B(0,\rho)$, and suppose that the boundary condition imposed on $\partial \om$ is $u_0(x)=x$.  \\
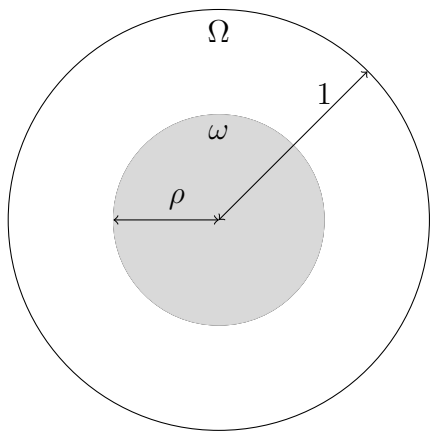
\begin{figure}[H]
\centering
\begin{minipage}{1\textwidth}
\centering
\begin{tikzpicture}[scale=1.4]
\node (M) at ( 0.0, 0.0) {}; 
\draw (M.center) circle (2.0cm);
\draw (M.center) circle (1.0cm);
\draw[gray!30!,fill] (0,0) circle(1.0cm);
\node[below] (o) at ( 0.0, 1.0) {$\omega$}; 
\node[below] (O) at ( 0.0, 2.0) {$\Omega$}; 
\draw[<->]        (0,0)   -- (-1.0,0);
\node[above] (M) at ( -0.4, 0.0) {$\rho$}; 
\draw[<->]        (0,0)   -- (1.41,1.41);
\node[above] (M) at ( 1.0, 1.0) {$1$}; 
\end{tikzpicture}
\end{minipage}
\caption{Illustration of the disk-disk problem for $\rho=0.5$.}
\label{problem:disk-disk}
\end{figure}
Then, by applying Proposition \ref{howtosolveEL}, we calculate that the function
\begin{align}
    u(x):=\left\{\begin{array}{l l} \zeta x & \ \ x \in \omega \\
    \left(\xi+\frac{1-\xi}{|x|^2}\right)x & \ \ x \in \om\setminus\omega \end{array}\right.\label{diskdisksol}
\end{align}
obeys conditions (i) and (ii) of Proposition \ref{howtosolveEL} provided
\begin{align*}
    \zeta:=\frac{4}{4+M-M\rho^2}, && \xi:=\frac{4+M}{4+M-M\rho^2}.
\end{align*}

In the course of the calculation above we made use of Proposition \ref{continuity} to require that the solution is, in particular, continuous across $\partial \omega$.  In order to satisfy the mean coercivity hypothesis of Proposition \ref{continuity}, it is sufficient to assume that $|M| < 4$, as we show in Lemma \ref{mcisland} below.  Before that, we remark that the solution $u$ given by \eqref{diskdisksol} is valid for all $M>0$, not just those that through an application of Lemma \ref{mcisland} render $F$ mean coercive.  Presumably in these `large M' cases $u$ is a continuous stationary point of $F$ but is not a minimizer.

\begin{lemma}\label{mcisland}
  The functional 
$$F(\varphi)=\int_{\om}|\nabla \varphi|^2 + M \chi_{\omega} \det \nabla \varphi \, \dx$$
is mean coercive on $H_0^1(\om;\R^2)$ if $|M| < 4$.    
\end{lemma}
\begin{proof}
   Let $|M|<4$ and write
    \begin{align*}
        F(\varphi) & = \eps\int_{\om}|\nabla \varphi|^2 \, \dx + (1-\eps)\int_{\om}|\nabla \varphi|^2 + \frac{M}{1-\eps} \chi_{\omega} \det \nabla \varphi \, \dx, 
    \end{align*}
where, by \cite[Proposition 3.4]{BKV23}, the integral functional with prefactor $1-\eps$ is nonnegative on $H_0^1(\om;\R^2)$ if and only if $|\frac{M}{1-\eps}|\leq 4$. Given that $|M|<4$, this condition is easily satisfied by choosing $\eps>0$ sufficiently small. Hence $F$ is mean coercive. 
\end{proof}
\begin{remark}\label{munktell}\emph{Example \eqref{diskdisksol} illustrates a number of points, inlcuding that:
\begin{itemize}\item[(a)] the solution $u$ is not $C^1$, and nor could it be since it would then necessarily be harmonic throughout $\om$, and hence, in view of the boundary conditions, equal to the identity throughout the domain, in clear violation of condition (ii) of Proposition \ref{howtosolveEL}, and
\item[(b)] the Jacobian $\det \nabla u$ is radial, discontinuous and obeys 
\begin{displaymath}
    \det \nabla u(x)=\left\{\begin{array}{l l} \zeta^2  & \ \ x \in \omega \\
    \xi^2-(1-\xi)^2 |x|^{-4} & \ \ x \in \om\setminus\omega \end{array}\right.
\end{displaymath}
In particular, $\det \nabla u$ jumps `up' as $\partial \omega$ is crossed from inside to out by an amount 
\begin{align}\label{jumpCL}
    \frac{8M}{(M\rho^2-M-4)^2}, 
\end{align}
presumably reflecting the fact that, when minimizing the energy $F$ defined in \eqref{Finhom}, it is better to have a smaller Jacobian in regions where the term $M\chi_{\omega} \det \nabla u$ is `active' and $M>0$. \end{itemize}}  
\end{remark}
By inspection, we deduce from \eqref{jumpCL} that the jump in $\det \nabla u$ across $\partial \omega$ is of size $\frac{M}{2}|\partial_{\tau} u|^2$,
which, as we will now see, is not a coincidence provided we make certain assumptions about the normal and tangential derivatives of $u$ on $\partial \omega$.  A priori, we do not even know whether the functions $\partial_\nu u\arrowvert_{\omega}$, $\partial_\nu u\arrowvert_{\om\setminus \omega}$ and $\partial_{\tau}u$ exist pointwise on the (1-dimensional) set $\partial \om$.  But for the purposes of the following formal argument, let us assume that $u$  obeys
\begin{align}\label{strong}  2\partial_\nu u\arrowvert_{\omega}+ 2 \partial_{-\nu} u\arrowvert_{\Omega\setminus \omega} - M J \partial_{\tau}u = 0 \quad \sch^1-a.e. \ \textrm{on} \ \partial \omega \setminus \partial \om\end{align}
and also that 
\begin{align}\label{shift1}\det \nabla u\arrowvert_{\omega}(x) & = \partial_{\tau}u(x)  \cdot J \partial_{\nu}u\arrowvert_{\omega}(x) \quad \mathrm{and}, \\
\label{shift2}\det \nabla u\arrowvert_{\om \setminus \omega}(x) & = \partial_{\tau}u(x) \cdot J \partial_{\nu}u\arrowvert_{\om \setminus \omega}(x),
\end{align} 
except possibly for an $\sch^1-$null subset of $\partial \omega$.  The origin of \eqref{shift1} and \eqref{shift2} lies in the identity $\det \nabla u = \partial_{\tau}u \cdot J \partial_{\nu}u$, which holds a.e.\@ with respect to 2-dimensional Lebesgue measure.  The strengthening we assume is that this holds $\sch^1-$a.e.\@ on $\partial\omega$. Under the circumstances just outlined, we claim that for $\sch^1-$a.e. $x \in \partial \omega$ it holds that 
\begin{align}\label{CL}
   \det \nabla u\arrowvert_{\om \setminus \omega}(x) - \det \nabla u\arrowvert_{\omega}(x) = \frac{M}{2}|\partial_{\tau}u(x)|^2.
\end{align}
This is easily proved:  apply $J$ to both sides of \eqref{strong} and recall that $J^2=-\1$ to obtain for $\sch^1-$a.e. $x$ in $\partial \omega$
\begin{align*}
    J \partial_\nu u\arrowvert_{\omega} - J \partial_{\nu} u\arrowvert_{\Omega\setminus \omega} + \frac{M}{2}  \partial_{\tau} u= 0.
\end{align*}
Taking the inner product of both sides with $\partial_{\tau}u$, applying \eqref{shift1}
and \eqref{shift2}, and then rearranging slightly gives
 \eqref{CL}.

\begin{remark}\emph{
We can further infer from Remark \ref{munktell} (b) that the abrupt change in the Jacobian is `uniformly spread' around the smooth set $\partial \omega$.  This is in contrast with cases in which the subdomain $\omega$ has `sharp corners', where numerical evidence suggests that the greatest jumps in the Jacobian occur non-uniformly.  See Section \ref{diskdisk} for the latter, and the discussion following \eqref{beeth5} for an analytic example.}
\end{remark}

\subsection{The case that $\omega$ is a sector of the unit disk in $\R^2$.}
\label{disk-sector}
Let $\omega$ be the sector of the unit disk $B$ defined by $|\theta| \leq \pi/4$ in plane polar coordinates a shown in Figure \ref{problem:disk-section}.
\begin{figure}[H]
\centering
\begin{minipage}{1\textwidth}
\centering
\begin{tikzpicture}[scale=1.4]
\node (M) at ( 0.0, 0.0) {}; 
\node[] (o) at ( 1.0, 0.0) {$\omega$}; 
\node[below] (O) at ( 0.0, 2.0) {$B$}; 
\draw[-]        (0,0)   -- (1.41,1.41);
\draw[-]        (0,0)   -- (1.41,-1.41);
\draw[gray!30!,fill] (0:0) -- (1.41,1.41) arc (45:-45:2) -- cycle;
        \node[] (o) at ( 1.0, 0.0) {$\omega$}; 
        \draw (M.center) circle (2.0cm);
\node[below] (O) at ( 0.0, 2.0) {}; 
\end{tikzpicture}
\end{minipage}
\caption{Illustration of the disk-sector problem.  Here, $B$ is the unit disk in $\R^2$.}
\label{problem:disk-section}
\end{figure}
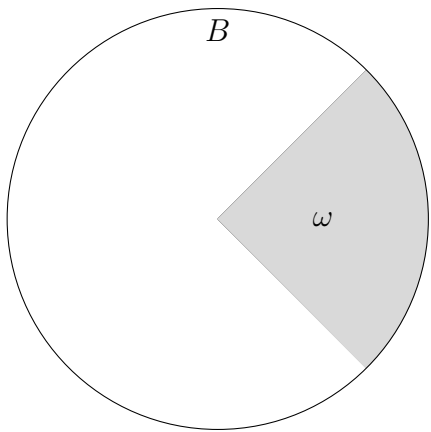
Then a concrete solution to the Euler-Lagrange equation as set out in Proposition \ref{howtosolveEL} is:
\begin{align}
    u(R,\theta)& = \left\{\begin{array}{l l} u^s(R,\theta) & \mathrm{in} \ \omega, \\
    u^p(R,\theta) & x \in B \setminus\omega,
\end{array}\right. \label{example1}
\end{align}
where 
\begin{align}\label{beeth5}
u^s(R,\theta) = \left(\begin{array}{c} 1 \\ R^2 \sin(2\theta) 
    \end{array}\right)
    \quad \text{and} \quad
    u^p(R,\theta) = \left(\begin{array}{c} 1 +\frac{M}{2}R^2\cos(2\theta)\\ R^2 \sin(2\theta) 
    \end{array}\right).
    \end{align}
The form of this solution is taken from Proposition \ref{tango} below.  We remark that since the normal and tangential derivatives clearly exist along $\partial \omega$, with the possible exception of the origin, the argument leading to \eqref{CL} is valid, and hence the jump in the Jacobian $\det\nabla u$ across $\partial \omega$ is given by 
$$\frac{M}{2}|\partial_{\tau}u|^2 = 2R^2,$$
which, we note, is maximal as $\partial B$ is approached.

The solution in \eqref{example1}
is a particular case of the following general form of solution that applies to boundary data $u_0$ in $H^{-\frac{1}{2}}(\partial B,\R^2)$ such that  
\begin{enumerate}
\item[(a)] 
$u_0$ obeys the symmetry condition 
\begin{align}\label{bcsymm}
    u(x) = E u_0(Ex) \quad x \in \partial B,
\end{align}
where $E$ is the $2 \times 2$ matrix 
\begin{align}\label{E}
    E = \left(\begin{array}{c c} 1 & 0 \\ 0 & -1\end{array}\right).\end{align}
\item[(b)] in terms of plane polar coordinates on $\partial \omega$, $u_0$ has a development of the type 
\begin{align*}
u_0(1,\theta) =\left\{ \begin{array}{ll}
    \left( \begin{array}{c} \sum_{k \geq 0} A_{4k}\cos(4k\theta)+A_{4k+2}\cos((4k+2)\theta) \\
   \sum_{k \geq 0} B_{4k}\sin(4k\theta)+B_{4k+2}\sin((4k+2)\theta)
    \end{array} \right) & \mathrm{if} \ (1,\theta) \in \partial \omega \\
 \left(\begin{array}{c} \sum_{k \geq 0} A_{4k}\cos(4k\theta)+\left(A_{4k+2}+\frac{M}{2}B_{4k+2}\right)\cos((4k+2)\theta) \\
   \sum_{k \geq 0} \left(B_{4k}+\frac{M}{2}A_{4k}\right)\sin(4k\theta)+B_{4k+2}\sin((4k+2)\theta)
   \end{array}\right) & \mathrm{if}\  (1,\theta) \in \partial B \setminus \partial \omega
\end{array}\right.
\end{align*}
\end{enumerate}
When $u_0$ satisfies conditions (a) and (b), we refer to $u_0$ as being \emph{suitably prepared}. 

\begin{proposition}\label{tango}
Let $F(u)$ be given by 
\begin{align*}
    F(u)=\int_{B} |\nabla u|^2 + M\chi_{\omega} \det \nabla u \, \dx,
\end{align*}
where $\omega$ is the sector defined by $|\theta| \leq \frac{\pi}{4}$ in plane polar coordinates and $|M|<4$.  Assume that $u_0$ is suitably prepared boundary data.  Then the unique minimizer of $F(u)$ in the class $H_{u_0}^1(B;\R^2)$
obeys $u(x)=Eu(Ex)$ for almost every $x \in B$.  Moreover, in plane polar coordinates, $u$ has the formal representation \begin{align}
  \label{sector} \left( \begin{array}{c} u_1^s(R,\theta) \\
   u_2^s(R,\theta)
    \end{array} \right)  &  = \left( \begin{array}{c} \sum_{k \geq 0} A_{4k}\cos(4k\theta)R^{4k}+A_{4k+2}\cos((4k+2)\theta)R^{4k+2} \\
   \sum_{k \geq 0} B_{4k}\sin(4k\theta)R^{4k}+B_{4k+2}\sin((4k+2)\theta)R^{4k+2}
    \end{array} \right),
\end{align}
valid for $(R,\theta)$ corresponding to the sector $\omega$, and
\begin{align}\label{pacman}
   \left( \begin{array}{c} u_1^p(R,\theta) \\
   u_2^p(R,\theta)
    \end{array} \right)  &  = \left( \begin{array}{c} \sum_{k \geq 0} A_{4k}\cos(4k\theta)R^{4k}+\left(A_{4k+2}+\frac{M}{2}B_{4k+2}\right)\cos((4k+2)\theta)R^{4k+2} \\
   \sum_{k \geq 0} \left(B_{4k}+\frac{M}{2}A_{4k}\right)\sin(4k\theta)R^{4k}+B_{4k+2}\sin((4k+2)\theta)R^{4k+2}
    \end{array} \right)
\end{align}
otherwise.  
\end{proposition}
\begin{proof}
We identify the ball $B$ with the set $\{(R,\theta): \ 0 \leq R < 1, \ -\pi < \theta \leq \pi\}$. Defining $\bar{u}(x):=Eu(Ex)$ for all $x \in B$, we find by a direct calculation that 
$$F(u)=F(\bar{u})$$
and hence, by uniqueness, that $u(x)=\bar{u}(x)$ for almost every $x \in B$.  This proves the first part of the statement of the proposition which, in components, amounts to
\begin{align}\label{symm}
    \left(\begin{array}{c}
       u_1(x_1,x_2)     \\
       u_2(x_1,x_2)
          \end{array}\right) & = \left(\begin{array}{c}
       u_1(x_1,-x_2)     \\
       -u_2(x_1,-x_2)
          \end{array}\right).
\end{align}
Hence $u_1$ is an even function of $x_2$ and $u_2$ is odd in $x_2$.  Given that $u$ solves the Euler-Lagrange equation in Proposition \ref{howtosolveEL}, it must in particular be that $u$ is harmonic in both $\omega$ and $B \setminus \omega$.  It is standard that solutions to Laplace's equation can be expressed as superpositions of functions of the form $R^{\alpha} \cos(\alpha \theta)$ and $R^{\alpha} \sin(\alpha \theta)$, and in view of the fact that $u_1$ is even in $x_2$, it is clear that in each of $\omega$ and $B \setminus \omega$, $u_1$ should depend only on (sums of) functions of the type $v(R,\theta;\alpha):=R^{\alpha}\cos(\alpha \theta)$, with a similar outcome for the form of $u_2$.   

The region $B\setminus\omega$ is cut by the line $\theta=-\pi$, which, in our coordinate system, is equivalent to $\theta=\pi$.  Letting $(B\setminus\omega)^+$ be the part of $B\setminus \omega$ characterized by polar angles in the interval $(\pi/4,\pi]$, and by identifying $(B\setminus\omega)^-$ similarly with polar angles belonging to $[-\pi,-\pi/4)$, we find that the function $\theta \mapsto v(R,\theta;\alpha)$ is smooth on $B\setminus \omega$ only if 
\begin{align}\label{alphaint}
   \lim_{\theta \to \pi_{-}} \partial_{\theta}v(R,\theta;\alpha)\arrowvert_{(B\setminus\omega)^+} & = \lim_{\theta \to -\pi_{+}} \partial_{\theta}v(R,\theta;\alpha)\arrowvert_{(B\setminus\omega)^-}.
\end{align}
Equation \eqref{alphaint} then implies that $\alpha \in \mathbb{Z}$, and hence
\begin{align}\label{u1p} u_1^p =\sum_{j=0}^{\infty} C_j R^j \cos(j\theta) \quad \quad (R,\theta) \in B \setminus \omega,\end{align}
and, similarly, 
\begin{align}\label{u2p}u_2^p=\sum_{j=0}^{\infty} D_j R^j \sin(j\theta) \quad \quad (R,\theta) \in B \setminus \omega.\end{align}  
By Proposition \ref{continuity}, $u$ must be continuous in $B$, which in particular means that we may treat $u_1(R,\pi/4)$ as a boundary condition when solving $\Delta u_1^s=0$ in $\omega$.  It follows that 
$$u_1^s =\sum_{j=0}^{\infty} A_j R^j \cos(j\theta) \quad \quad (R,\theta) \in \omega,$$
 where, by a matching argument, it is necessary that $A_j=C_j$ for all $j$ that are not of the form $j=4n+2$ for some nonnegative integer $n$.  Similarly, 
 $$u_2^s =\sum_{j=0}^{\infty} B_j R^j \sin(j\theta) \quad \quad (R,\theta) \in \omega,$$
where it is necessary that $B_j=D_j$ for all $j$ that are not of the form $j=4n$ for some nonnegative integer $n$.   
 
 To conclude the proof of the proposition, we show that the final form of the solution, as given by \eqref{sector} and \eqref{pacman}, flows from the hitherto unused `jump condition' part of the Euler-Lagrange equation, namely \eqref{conslawboundary}.  In the current coordinates, when calculated along the upper part of $\partial \omega$, \eqref{conslawboundary} becomes 
 \begin{align}\label{JC}
 2 \partial_{\theta} u^s(R,\pi/4) - 2 \partial_{\theta} u^p(R,\pi/4)  = MR J^T \partial_{R} u^s(R,\pi/4) \quad \quad 0<R<1.
 \end{align}
 The $e_1$ component reads
\begin{align*}
    \sum_{j=0}^{\infty} j (C_j - A_j) R^j \sin(j \pi/4) = \frac{M}{2} \sum_{j=0}^{\infty} j B_j R^j \sin(j \pi /4).
\end{align*}
The only possible non-zero terms on the left-hand side correspond to $j$ of the form $j=4k+2$, since in all other cases we have $A_j=C_j$.  Thus in any group of four consecutive integers $4k, \ldots, 4k+3$, where $k\geq 0$, it must be, by a straightforward matching argument, that $B_{4k+1}=B_{4k+3}=0$ and 
$$ C_{4k+2}-A_{4k+2}= \frac{M}{2}B_{4k+2}.$$
Hence $D_{4k+1}=D_{4k+3}=0$ and, by studying the $e_2$ component of \eqref{JC}, we find that $A_{4k+1}=A_{4k+3}=0$, so $C_{4k+1}=C_{4k+3}=0$, and
$$B_{4k}-D_{4k}= -\frac{M}{2}A_{4k}.$$
Eliminating $C_{j}$ and $D_{j}$ from \eqref{u1p} and \eqref{u2p} leads to \eqref{pacman}.  Finally, the symmetry of the solution $u$ expressed via \eqref{symm} implies in particular that $\partial_{\theta}u(R,\pi/4)=-E \partial_{\theta}u(R,-\pi/4)$ and 
$\partial_{R}u(R,\pi/4)=E\partial_{R}u(R,-\pi/4)$, where $E$ is given by \eqref{E}.  Inserting this into \eqref{JC} gives, after some manipulation using the facts that $E^2=\1$ and $EJE=-J$,
\begin{align*}
 2 \partial_{\theta} u^s(R,-\pi/4) - 2 \partial_{\theta} u^p(R,-\pi/4)  & = MR EJE \partial_{R} u^s(R,-\pi/4) \\
 & = -MR J\partial_{R} u^s(R,-\pi/4)
 \end{align*}
for $0<R < 1$.  It can be checked that this is exactly \eqref{conslawboundary} when applied to the lower part of $\partial \omega$, and hence this is satisfied whenever \eqref{JC} holds.  The solution fits the suitably prepared data $u_0$ by construction.  
\end{proof}

In fact, we believe the previous result holds for general boundary data $u_0$ in $H^{-\frac{1}{2}}(\partial B,\R^2)$ and not just for the suitably prepared kind.  Indeed, no such restriction is needed in the variational principle that leads to the existence of $u$ minimizing $F(\cdot)$, so why should it appear as a condition in Proposition \ref{tango}?  A fortiori, when $|M|<4$ we could infer---again directly from the variational principle---that in order to match the solution given in \eqref{sector} and \eqref{pacman}, any $u_0$ should have a unique development given by condition (b) above.   There are several levels of complexity to this problem, perhaps the most basic of which is, given $u_0$, to find a way to compute for nonnegative integers $k$ the coefficients $A_{4k}$, $A_{4k+2}$, $B_{4k}$, $B_{4k+2}$ appearing in \eqref{sector}
and \eqref{pacman}.  Here is one practical approach that rephrases the relation $u=u_0$ on $\partial B$ in terms of finding extensions to the various component functions $u_{01}(1,\theta)$ and $u_{02}(1,\theta)$.  We must stress that, for general boundary data $u_0$, while our method shows that these extensions exist and are unique, it does not show \emph{how} to find them. 

Let
\begin{align}\label{division}
    u_0(1,\theta) = \left\{\begin{array}{ll}
        u_0^s(1,\theta) & |\theta| \leq \pi/4 \\
        u_0^p(1,\theta) & \pi/4 \leq |\theta| \leq \pi
    \end{array}\right.
\end{align}
and consider, for illustration, the problem of fitting the first components $u_{01}^s(1,\theta)$ and $u_{01}^p(1,\theta)$ to the solution $u$ given in Proposition \ref{tango}.  Let $\overline{u_{01}^s}(1,\theta)$ be any even extension of $u_{01}^s(1,\theta)$ to the interval $[-\pi/2,\pi/2]$ and suppose that we seek $A_{4k}, A_{4k+2}$ for $k\geq 0$ such that 
\begin{align}\label{fairestisle}
    \overline{u_{01}^s}(1,\theta) = \sum_{k \geq 0} A_{4k}\cos(4k\theta)+A_{4k+2}\cos((4k+2)\theta) \quad \quad |\theta|\leq \frac{\pi}{2}.
\end{align}
Setting $\Theta=2\theta$, \eqref{fairestisle} is equivalent to 
\begin{align*}
    \overline{u_{01}^s}\left(1,\frac{\Theta}{2}\right) = \sum_{k \geq 0} A_{4k}\cos(2k\Theta)+A_{4k+2}\cos((2k+1)\Theta) \quad \quad |\Theta|\leq \pi,
\end{align*}
from which it is immediate that $\{A_{4k},A_{4k+2}\}_{k \geq 0}$ are the Fourier cosine coefficients of $\overline{u_{01}^s}\left(1,\frac{\Theta}{2}\right)$ and, moreover, by restriction, that the desired fitting 
\begin{align}\label{barber}
    u_{01}^s(1,\theta) = \sum_{k \geq 0} A_{4k}\cos(4k\theta)+A_{4k+2}\cos((4k+2)\theta) \quad \quad |\theta|\leq \frac{\pi}{4}
\end{align}
has been achieved.  Note the apparent `degrees of freedom':  there are potentially infinitely many choices of coefficients $\{A_{4k},A_{4k+2}\}_{k \geq 0}$ that are consistent with \eqref{barber}. 

The procedure for fitting a series of the form given by the first component of \eqref{pacman}, evaluated at $R=1$, to $u_{01}^p(1,\theta)$ is similar, but there are more restrictions.  Let $\overline{u_{01}^p}(1,\theta)$ be an even extension of $u_{01}^p(1,\theta)$ from $\{\theta: \ \frac{\pi}{4} \leq |\theta|\leq \pi\}$ to $[-\pi,\pi]$.  It suffices to find coefficients $\{P_n\}_{n \geq 0}$ such that 
\begin{align*}
    \overline{u_{01}^p}(1,\theta) = \sum_{n \geq 0} P_n \cos(n \theta)   \quad \quad |\theta|\leq \pi
\end{align*}
and to impose $P_{4k+1}=P_{4k+3}=0$ through the choice of the extension, as well as $P_{4k}=A_{4k}$ and $P_{4k+2}=B_{4k+2}+\frac{M}{2}A_{4k+2}$, where $A_{4k+2}$ is as above and $B_{4k+2}$ is yet to be defined.  (See Proposition \ref{bigcountry} for the latter.)  Assuming this has been done, by restriction we then have
\begin{align*}
u_{01}^p(1,\theta) = \sum_{k\geq 0} A_{4k}\cos(4k\theta) + \left(B_{4k+2}+\frac{M}{2}A_{4k+2}\right)\cos((4k+2)\theta) \quad \quad \frac{\pi}{4} \leq |\theta|\leq \pi.
\end{align*}
On this occasion, the extension  $\overline{u_{01}^p}(1,\theta)$ is required to have no odd Fourier cosine modes, and is connected to extensions including  $\overline{u_{01}^s}$ via the requirement that 
\begin{align*}
    \int_{-\pi}^{\pi}  \overline{u_{01}^p}(1,\theta)\cos(4k\theta) \, d\theta = \int_{-\pi}^{\pi}  \overline{u_{01}^s}\left(1,\frac{\Theta}{2}\right)\cos(2k\Theta) \, d\Theta \quad \quad \quad k \geq 0.
\end{align*}
Other such conditions can be derived similarly, and the results are recorded in Proposition \ref{bigcountry} below, as is the observation that, despite the apparent latitude available to us in the choice of even (and odd, see below) extensions, the uniqueness of the minimizing $u$ forces the corresponding extensions to be unique. This is in fact easy to see:  since the minimizer $u$ is unique, the coefficients $A_{4k}, \ldots, B_{4k+2}$ are also unique, and hence so are the Fourier cosine and sine series defining the extensions $\overline{u_{01}^s}, \overline{u_{01}^p},\overline{u_{02}^s}$ and $,\overline{u_{02}^p}$.  

    \begin{proposition}\label{bigcountry}
    Let $u_0$ belong to $H^{-\frac{1}{2}}(\partial B,\R^2)$ and let $u_0^s$ and $u_0^p$ be given by \eqref{division}. \\
    \noindent {\bf{(a)}} Let $\overline{u_{01}^s}$ be any even extension of $u_{01}(1,\theta)$ to $|\theta|\leq \pi/2$ and 
$\overline{u_{02}^s}$ be any odd extension of $u_{01}(1,\theta)$ to $|\theta|\leq \pi/2$.  Let the Fourier cosine and sine series of $\overline{u_{01}^s}(1,\Theta/2)$ and $\overline{u_{02}^s}(1,\Theta/2)$ be 
\begin{align}\label{beldivedremo1}
\overline{u_{01}^s}\left(1,\frac{\Theta}{2}\right) & = \sum_{k \geq 0} a_{4k}\cos(2k\Theta)+a_{4k+2}\cos((2k+1)\Theta) \quad \quad |\Theta|\leq \pi \ \mathrm{and}\\ \label{beldivedremo2}
\overline{u_{02}^s}\left(1,\frac{\Theta}{2}\right) & = \sum_{k \geq 0} b_{4k}\cos(2k\Theta)+b_{4k+2}\cos((2k+1)\Theta) \quad \quad |\Theta|\leq \pi.
\end{align}
Then the function 
    \begin{align*}
 w(R,\theta) &  := \left( \begin{array}{c} \sum_{k \geq 0} a_{4k}\cos(4k\theta)R^{4k}+a_{4k+2}\cos((4k+2)\theta)R^{4k+2} \\
   \sum_{k \geq 0} b_{4k}\sin(4k\theta)R^{4k}+b_{4k+2}\sin((4k+2)\theta)R^{4k+2}
    \end{array} \right)
\end{align*}
defined for $0\leq R \leq 1$ and $|\theta|\leq \frac{\pi}{4}$ is such that $w=u_0$ on $\partial \omega$.  
\\
    \noindent {\bf{(b)}} 
    Let $\overline{u_{01}^p}(1,\theta)$ be an even extension of $u_{01}^p(1,\theta)$ from $\{\theta: \ \frac{\pi}{4} \leq |\theta|\leq \pi\}$ to $[-\pi,\pi]$.  Similarly, let $\overline{u_{02}^p}(1,\theta)$ be an odd extension of $u_{02}^p(1,\theta)$ from $\{\theta: \ \frac{\pi}{4} \leq |\theta|\leq \pi\}$ to $[-\pi,\pi]$.  Let the Fourier cosine and sine series of $\overline{u_{01}^p}(1,\theta)$ and $\overline{u_{02}^s}(1,\theta)$ be 
\begin{align}\label{bruchnot1}
\overline{u_{01}^p}\left(1,\theta\right) &= \sum_{k \geq 0} a'_{4k}\cos(2k\theta)+a'_{4k+2}\cos((2k+1)\theta) \quad \quad |\theta|\leq \pi \ \mathrm{and}\\
\label{bruchnot2}
\overline{u_{02}^p}\left(1,\theta\right) &= \sum_{k \geq 0} b'_{4k}\cos(2k\theta)+b'_{4k+2}\cos((2k+1)\theta) \quad \quad |\theta|\leq \pi.
\end{align}
Then the function 
    \begin{align*}
 W(R,\theta) &  = \left( \begin{array}{c} \sum_{k \geq 0} a'_{4k}\cos(4k\theta)R^{4k}+a'_{4k+2}\cos((4k+2)\theta)R^{4k+2} \\
   \sum_{k \geq 0} b'_{4k}\sin(4k\theta)R^{4k}+b'_{4k+2}\sin((4k+2)\theta)R^{4k+2}
    \end{array} \right)
\end{align*}
defined for $0\leq R \leq 1$ and $\{\theta: \ \frac{\pi}{4} \leq |\theta|\leq \pi\}$ is such that $W=u_0$ on $\partial B \setminus \partial \omega$.  
\\
    \noindent {\bf{(c)}} There are unique extensions $\overline{u_{01}^s}$, $\overline{u_{02}^s}$, $\overline{u_{01}^p}$ and $\overline{u_{02}^p}$
    such that the coefficients appearing in \eqref{beldivedremo1}, \eqref{beldivedremo2}, \eqref{bruchnot1} and \eqref{bruchnot2} are related by the equations
    \begin{align*}
    a'_{4k} & =a_{4k} \\
       a'_{4k+2} & =a_{4k+2}+\frac{M}{2} b_{4k+2}  \\
       b'_{4k} & =b_{4k}+\frac{M}{2}a_{4k} \\
       b'_{4k+2} & =b_{4k+2}
    \end{align*}
    for $k \geq 0$.    In these circumstances, the unique global minimizer $u$ of $F(u)$ in $H_{u_0}^1(B;\R^2)$ is given by 
    \begin{align*}
        u(R,\theta) = \left\{ 
        \begin{array}{ll }
          w(R,\theta)   &  \quad (R,\theta) \in \omega\\
            W(R,\theta) &  \quad (R,\theta) \in B \setminus \omega.
        \end{array}
    \right.
    \end{align*}
    \end{proposition}

\subsection{An `island problem' in three dimensions} \label{3dislandexample}

In this subsection we treat the functional \eqref{island3d} given by 
 \begin{align}\label{3dF}
     \tilde{F}(u):=\int_\om |\nabla u|^2 + T \cdot \cof \nabla u \dx,
 \end{align}
 where $\om$ is a given domain in $\R^3$ and where $T \in L^{\infty}(\om,\R^{3 \times 3})$ is given by $T:=T_0 \chi_{_{\omega}}$ for some constant matrix $T_0$ and a fixed $\omega \subset \om$.  
The objective is to examine the behaviour of $\tilde{F}$ on the class of test functions $C_c^{\infty}(\om;\R^3)$ and then on the class $H_{u_0}^1(\om;\R^3)$.
Since the integrand of $\tilde{F}$ is $2-$homogeneous, it is clear that if there is just one test function $u$ such that $\tilde{F}(u)<0$ then, via a simple scaling argument that makes use of the zero boundary conditions in force, $\tilde{F}$ is unbounded below.  Hence either 
\begin{align*}
    \inf\{\tilde{F}(u): \ u  \in C_c^{\infty}(\om;\R^3)\} = -\infty
\end{align*}
or 
\begin{align*}
\min\{\tilde{F}(u): \ u  \in C_c^{\infty}(\om;\R^3)\} = 0.
\end{align*}
Our first result gives a condition on $||T||_{\infty}$ which guarantees that the second of these two possibilities holds.

\begin{lemma}\label{3dislandpos}
Let \begin{align*}
     \tilde{F}(u)=\int_\om |\nabla u|^2 + T \cdot \cof \nabla u \dx,
 \end{align*}
 where $\om$ is a domain in $\R^3$, and let $T \in L^{\infty}(\om,\R^{3 \times 3})$ be given by $T:=T_0 \chi_{_{\omega}}$, as above.  Assume that 
 \begin{align}\label{simplebound}
     |T_0| \leq 2\sqrt{3}.
 \end{align}
Then $\tilde{F}(u) \geq 0$ for all $u$ in $C_c^{\infty}(\om;\R^3)$.  Moreover, if \eqref{simplebound} holds with a strict inequality, there is $\gamma>0$ such that 
\begin{align*}
     \tilde{F}(u) \geq \gamma \int_\om |\nabla u|^2 \dx \quad \forall u \in H^1_0(\om;\R^3).
\end{align*}
\end{lemma}
\begin{proof}We make use of the well-known fact that $\int_{\om}\cof \nabla u \, \dx=0$ for any test function $u$ and note that it immediately implies 
\begin{align*}
    \int_{\om} T_0 : \cof \nabla u \, \dx = 0 
\end{align*}
for any constant matrix $T_0$.  Hence,
\begin{align*}
    \tilde{F}(u) & =\tilde{F}(u)-\int_{\om}\frac{T_0}{2} : \cof \nabla u \, \dx \\ 
    & = \int_{\omega} \underbrace{|\nabla u|^2 + \frac{T_0}{2} : \cof \nabla u}_{(i)} \, \dx + \int_{\om \setminus \omega} \underbrace{|\nabla u|^2 - \frac{T_0}{2} : \cof \nabla u}_{(ii)} \, \dx.
\end{align*}
Applying Lemma \ref{symphonynumber9} below, the integrands indicated by $(i)$ and $(ii)$ are pointwise nonnegative as long as $|T_0|\leq 2\sqrt{3}$, which proves the first part of the proposition. 
Now assume that $|T_0| < 2\sqrt{3}$ and consider, for any $\gamma \in (0,1)$, 
\begin{align*}
    \hat{F}(u)  = \gamma\int_{\om} |\nabla u|^2 \, \dx + (1-\gamma)\underbrace{\int_{\om} |\nabla u|^2 + \frac{T_0}{1-\gamma} : \cof \nabla u \, \dx}_{(iii)}.
\end{align*}
Choosing $\gamma>0$ so that $|T_0|/(1-\gamma)\leq 2\sqrt{3}$, we can apply the result of the first part of the proposition to the functional labeled $(iii)$ and conclude that it is nonnegative.  This proves the second part of the proposition.
\end{proof}

The following straightforward technical lemma was needed in the proof of Proposition \ref{3dislandpos}. To keep the paper self-contained, we give a short proof but observe that the result is almost certainly available elsewhere in the literature.

\begin{figure}
    \centering
    \includegraphics[width=0.5\linewidth]{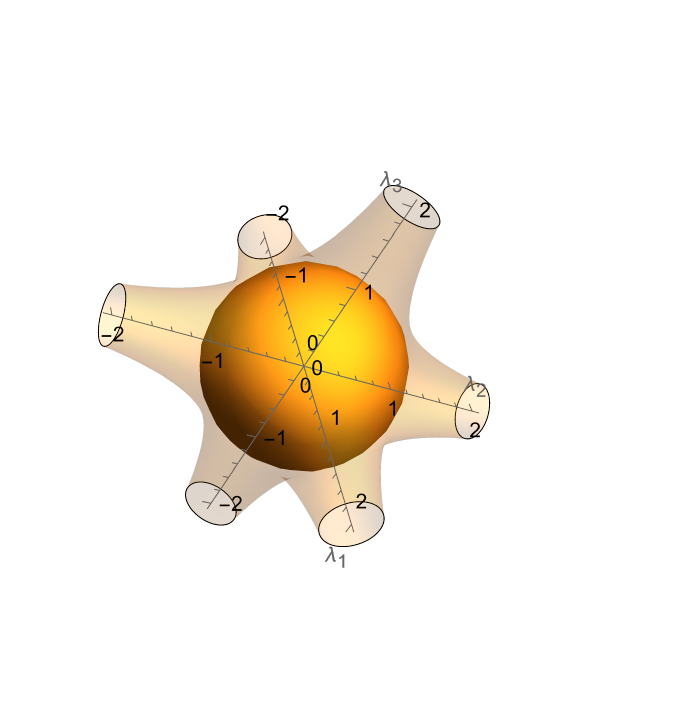}
    \caption{The unit ball $|A|^2=1$ corresponds to $\{(\lambda_1,\lambda_2,\lambda_3) \in R^3: \ \lambda_1^2 + \lambda_2^2 + \lambda_3^2=1\}$, where the $\lambda_i\geq 0$ are the singular values of $A$, while the light-coloured surface shown corresponds in these coordinates to $A$ such that $\sqrt{3}\, |\cof A| = 1$.  The inequality featuring in Lemma \ref{symphonynumber9} is equivalent to $\sqrt{3}\, |\cof A| \leq 1$ for all $A$ such that $|A|=1$.  If this inequality were to fail for some $A$ then there would exist a point $(\lambda_1,\lambda_2,\lambda_3)$ lying both strictly outside the surface shown and within the unit ball visible inside it, which, visually at least, is impossible.}
    \label{fig:enter-labelT}
\end{figure}

\begin{lemma}\label{symphonynumber9} Let $A\in \R^{3\times 3}$.  Then $\sqrt{3}\, |\cof A| \leq |A|^2$, and the inequality is sharp. 
\end{lemma}
 \begin{proof}Using a singular value decomposition for $A$, \cite[Prop 13.4]{Dac08} tells us that 
 \begin{align*}
     |A|^2 & = \lambda_1^2+\lambda_2^2+\lambda_3^2, \ \textrm{and}\\
     |\cof A|^2 &  =  \lambda_1^2\lambda_2^2 + \lambda_2^2\lambda_3^2+ \lambda_3^2\lambda_1^2.
 \end{align*}
 In these terms, an inequality of the form $\mu|\cof A| \leq |A|^2$, where we deliberately leave $\mu$ unspecified, 
is equivalent to 
\begin{align}\label{mostly}
 (\mu^2-2)(\lambda_1^2\lambda_2^2 + \lambda_2^2\lambda_3^2+ \lambda_3^2\lambda_1^2) \leq \lambda_1^4+\lambda_2^4+\lambda_3^4.   
\end{align}
It is easily checked that \eqref{mostly} holds only if $\mu^2-2\leq 1$, which implies that $\mu =\sqrt{3}$ is the largest possible.  To see that the stated inequality is sharp, take $A=\1$.\end{proof}

By means of Lemma \ref{3dislandpos}, one can prove the existence and uniqueness of a minimizer of $\tilde{F}$ in $\sca_2$, and that the associated Euler-Lagrange equation is linear in $u$. 

\begin{proposition}
    Let $\tilde{F}$ be given by \eqref{3dF}.  Then if $|T_0|<2\sqrt{3}$, $\tilde{F}$ has a unique global minimizer in $ H_{u_0}^1(\om;\R^3)$ which obeys the Euler-Lagrange equation
    \begin{align}\label{EL3d}
       \int_{\om} 2 \nabla u \cdot \nabla \varphi + T \cdot (\nabla u,\nabla \varphi) \, \dx = 0 \quad \forall \varphi \in H_0^1(\om;\R^3).
    \end{align}
    Here, given $A,B \in \R^{3 \times 3}$, $(A,B)$ is the $3 \times 3$ matrix with $(i,j)$ entry 
    \begin{align}\label{crossterm}
        (A,B)_{ij}=\eps^{iab}\eps^{jcd}A_{ac}B_{bd},
    \end{align}
    where $\eps^{iab}$ is the standard alternating symbol on three elements.\footnote{The alternating symbol appears in particular in the identity $(\cof A)_{ij} = 
    \frac{1}{2}\eps^{iab}\eps^{jcd}A_{ac}A_{bd}$, which explains the `cross term' $(A,B)$ in \eqref{crossterm}.}
  \end{proposition}
\begin{proof} The proof is similar to that of Proposition \ref{basicfact} and we omit most of the details other than to point out that, for $u \in H_{u_0}^1(\om;\R^3)$ and $\varphi \in H_0^1(\om,\R^3)$, the decomposition 
\begin{align*}
    \tilde{F}(u+\varphi)=\tilde{F}(u)+\int_{\om} 2\nabla u \cdot \nabla \varphi +  T \cdot  (\nabla u,\nabla \varphi) \, \dx + \tilde{F}(\varphi)
\end{align*}
    is the origin of the term $(\nabla u,\nabla \varphi)$ appearing in \eqref{EL3d}. 
\end{proof}

\section{The role of $F(u)$ in constrained variational problems}\label{section3}

Let $F(u)$ be given by \eqref{Finhom}, namely 
\begin{align}\label{Finhom2}
    F(u) = \int_{\om} |\nabla u|^2 + f \det \nabla u \dx
\end{align}
for some fixed $f$ belonging to $L^{\infty}(\om)$, 
and for any $u$ in $H^1(\om;\R^2)$ recall that
\begin{align}\label{dirichlet}\mathbb{D}(u) =\int_{\om}|\nabla u|^2 \dx.
\end{align}
It is a classical problem, whose origins lie in incompressible nonlinear elasticity theory, to minimize $\D(u)$ over functions $u$ such that $\det \nabla u = g$ a.e.\@, where $g$ is a fixed function.  By applying a boundary condition in the form of a trace, we let
\begin{align}\label{ag}
    \sca_g :=\{u \in H^1_{u_0}(\om;\R^2): \det \nabla u = g \ \mathrm{a.e.}\}.
\end{align}
The main result of this section, which we later illustrate by means of two detailed examples, is the following.
\begin{theorem}\label{minFgivesminD} Let $F(u)$ be given by \eqref{Finhom2} and let $\D(u)$ be the Dirichlet energy for $u$, as defined by \eqref{dirichlet}.  Assume that $F$ obeys the mean coercivity condition that there exists $\gamma > 0$ depending only the function $f$ and the domain $\om$ such that
\begin{align}\label{meancoercive}
    F(\varphi) \geq \gamma \int_{\om} |\nabla \varphi|^2 \dx \quad \quad \forall \varphi \in H_0^1(\om;\R^2).
\end{align}
Let $u$ minimize $F$ in $H^1_{u_0}(\om;\R^2)$. Then $u$ is the unique minimizer of $\D$ in the class $\sca_g$, where $g:=\det \nabla u$. 
\end{theorem}
The point is that by minimizing $F$ on the larger class $H^1_{u_0}(\om;\R^2)$, one can apply some of the machinery introduced in Section \ref{sectiontwo}, and there emerges a technique for generating minimizers of $\D$ on sets of constrained admissible functions $\sca_g$ as outlined in Steps 1-4\@ in the Introduction.  The proof of Theorem \ref{minFgivesminD} relies in part on the following decomposition result for $\D(u)$ in the class $\sca_g$,  which we remark is much like that of \eqref{MKdecomposition} for $F(u)$ in the class $H_{u_0}^1(\om;\R^2)$.   

\begin{lemma}\label{diddly} Let $\sca_g$ be given by \eqref{ag} and let $u,v \in \sca_g$.  Then 
\begin{align}\label{forall}
    \D(v) = \D(u) + \sco(u,v-u) + F(v-u),\end{align}
where 
\begin{align*}
    \sco(u,\varphi):=\int_{\om}2 \nabla u  \cdot \nabla \varphi + f \,\cof \nabla u \cdot \nabla \varphi \dx. 
\end{align*}
\end{lemma}
\begin{proof} From Proposition \ref{basicfact}, we have for any $u,v \in \sca_g$
\begin{align}\label{d774} F(v) & = F(u) + a(u,v-u)+ F(v-u).
\end{align}
Since $u,v \in \sca_g$, it follows that 
$$\int_{\om} f \, \det \nabla u \dx = \int_{\om} f \, \det \nabla v \dx, $$
and hence that 
$$ F(v) - F(u) = \D(v)-\D(u).$$
Substituting this into \eqref{d774} gives \eqref{forall}.
\end{proof}

Now we are able to give the proof of Theorem \ref{minFgivesminD}.

\begin{proof} Since $u$ minimizes $F(u)$ in $H^1_{u_0}(\om;\R^2)$, \eqref{ELinhombc} holds and we deduce that 
\begin{align*} \sco(u,v-u)=0
\end{align*}
for any $v$ in $H^1_{u_0}(\om;\R^2)$.   Further, since $u$ is assumed to belong to $\sca_g$, Lemma \ref{diddly} gives for any $v \in \sca_g$
\begin{align}\label{yap}
    \D(v)=\D(u) + F(\varphi)
\end{align}
where $\varphi:=v-u$ belongs to $H_0^1(\om;\R^2)$.  Finally, by mean coercivity \eqref{meancoercive}, we see that $\D(v) \geq \D(u)$ for all $v$ in $\sca_g$, with equality if and only if $v=u$ a.e..  
\end{proof}

We remark that Theorem \ref{minFgivesminD} applies to any of the solutions of the Euler-Lagrange equation for the functional 
$$F(u)= \int_{B} |\nabla u|^2 + M\chi_{\omega}\det \nabla u \dx$$
studied in Sections \ref{disk-disk} and \ref{disk-sector},
including those given by \eqref{diskdisksol} and \eqref{example1}, say, when $|M|<4$.  Since $F$ is mean coercive (by Lemma \ref{mcisland}), we can conclude that each of these solutions is a global minimizer of $F$ in a class of the form 
$\sca_{u_0}$ for suitable boundary data $u_0$.  Here, $\omega$ would either be a disk or a sector, as per Sections \ref{disk-disk} and \ref{sector} respectively.   

In the following two sections we apply Theorem \ref{minFgivesminD} to pressure functions $f$ that reflect rather different geometries.

\subsection{Example 1: the pure insulation problem with piecewise affine boundary conditions. }\label{insulation_f}

In this example our goal is to apply Theorem \ref{minFgivesminD} to the pressure function 
\begin{align}\label{threefi}
    f : =\sum_{i=1}^3 f_i \chi_{\omega_i}
\end{align}
in the rectangular domain $\om:=\omega_1 \cup \omega_2 \cup \omega_3$, where $\omega_1, \omega_2, \omega_3$ are specified in Figure \ref{pic:pureinsulation}.
\begin{figure}[H]
\centering
\begin{minipage}{0.69\textwidth}
\centering
\begin{tikzpicture}[scale=0.8]
\node (A) at (-4,-2) {}; 
\node[right=2 of A.center] (B) {};
\node[right=4 of A.center] (O) {};
\node[right=4 of B.center] (C) {};
\node[right=2 of C.center] (D) {};

\node[above=4 of A.center] (A') {};
\node[above=4 of B.center] (B') {};
\node[above=4 of O.center] (O') {};
\node[above=4 of C.center] (C') {};
\node[above=4 of D.center] (D') {};

\node[right=2 of B.center] (BC) {};
\node[right=2 of B'.center] (B'C') {};

\filldraw[thick, top color=gray!30!,bottom color=gray!30!] (A.center) rectangle node{} (B'.center);
\filldraw[thick, top color=gray!30!,bottom color=gray!30!] (C.center) rectangle node{} (D'.center);


\node (R1) at ($(A)!0.5!(B')$) {$\omega_1$};    
\node (R-1) at ($(B)!0.5!(C')$) {$\omega_2$};    
\node (R-1) at ($(C)!0.5!(D')$) {$\omega_3$};    


\draw[thick] (B.center) -- (C.center);
\draw[thick] (B'.center) -- (C'.center);
\end{tikzpicture}
\end{minipage}
\begin{minipage}{0.29\textwidth}
\begin{align*}
\omega_1  &=(\scalebox{0.9}{$-1,-\frac{1}{2}$})\times (\scalebox{0.9}{$-\frac{1}{2  }$},\scalebox{0.9}{$\frac{1}{2}$}), \\
     \omega_2  &=(-\scalebox{0.9}{$\frac{1}{2}$},\scalebox{0.9}{$\frac{1}{2}$})\times (\scalebox{0.9}{$-\frac{1}{2  }$},\scalebox{0.9}{$\frac{1}{2}$}), \\
     \omega_3  &=(\scalebox{0.9}{$\frac{1}{2}$},1)\times (\scalebox{0.9}{$-\frac{1}{2}$},\scalebox{0.9}{$\frac{1}{2}$}).
\end{align*}
\end{minipage}
\caption{Distribution of subdomains $\omega_1, \omega_2, \omega_3$ in Example 1.
}  \label{pic:pureinsulation}
\end{figure}
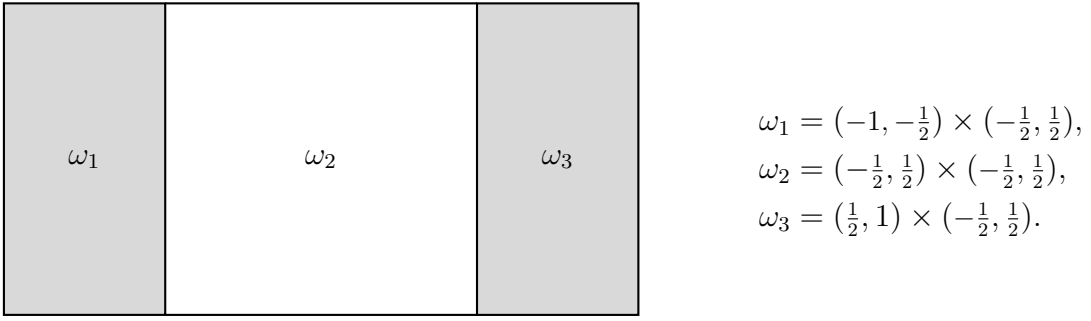

To ensure that the mean coercivity condition \eqref{meancoercive} holds, restrictions on the constants $f_1,f_2,f_3$ are necessary.

\begin{lemma}\label{MCexample1} Let $f$ be given by \eqref{threefi}.
Then there is $c>2$ such that the functional $F$ given by \eqref{Finhom} is mean coercive in the sense of \eqref{meancoercive} provided
\begin{align}\label{cond1} f_2 & = \frac{f_1+f_3}{2} \ \ \mathrm{and}\\ \label{cond2}
    |f_2& -f_1|  < c.
\end{align}
\end{lemma}
\begin{proof} Since for any smooth test function $\varphi$ we have $\int_{\om}\det \nabla \varphi \, \dx =0$, subtracting $f_2 \int_{\om}\det \nabla \varphi \, \dx$ from $F(\varphi)$ does not change its value.  Hence
\begin{align*}
    F(\varphi) = \int_{\om} |\nabla \varphi|^2 + \left((f_1-f_2)\chi_{\omega_1} +(f_3-f_2)\chi_{\omega_3}\right)\det \nabla \varphi \, \dx, 
\end{align*}
which, thanks to \eqref{cond1}, is of the special form
\begin{align}
    F(\varphi)=\int_{\om} |\nabla \varphi|^2 + \left(\sigma\chi_{\omega_1} -\sigma \chi_{\omega_3}\right)\det \nabla \varphi \, \dx 
\end{align}
and where $\sigma:=f_1-f_2$.  Now we `borrow' some of the Dirichlet term in order to prove mean coercivity, as follows:
\begin{align}\label{modifiedF}
F(\varphi) = \eps \int_{\om} |\nabla \varphi|^2 \, \dx + (1-\eps)\underbrace{\int_{\om}|\nabla \varphi|^2 + \left(\frac{\sigma}{1-\eps}\chi_{\omega_1} -\frac{\sigma}{1-\eps} \chi_{\omega_3}\right)\det \nabla \varphi \, \dx}_{=:K(\varphi)}
    \end{align}
By \eqref{cond2}, $\sigma:=f_1-f_2$ obeys $|\sigma| < c$, and so $|\frac{\sigma}{1-\eps}| < c$ also holds for sufficiently small $\eps$.  Hence, by \cite[Proposition 4.5]{BKV23}, the functional $K(\varphi)\geq 0$ for all test functions $\varphi$, which, together with $\eqref{modifiedF}$ gives the conclusion.  
\end{proof}

\begin{remark}\emph{Some degree of variation on the conditions \eqref{cond1} and \eqref{cond2} is possible whilst retaining the mean coercivity, but the case presented is the clearest we could find in the context of what in \cite{BKV23} is referred to as a `pure insulation' problem.}\end{remark}

Next, let $u_0$ be a continuous, piecewise affine function whose gradient $\nabla u_0$ obeys 
\begin{align}\label{clipper}
\nabla u_0(x) = \left\{\begin{array}{l l} A_1 & \textrm{if } x \in \omega_1 \\
A_2 & \textrm{if } x \in \omega_2 \\
A_3 & \textrm{if } x \in \omega_3,\end{array}\right.
\end{align}
where $A_1,A_2,A_3$ are $2 \times 2$ matrices to be chosen shortly.  Define $g:\om  \to \R$ by setting 
\begin{align}\label{g}
    g:=\sum_{i=1}^3 \det A_i \chi_{\omega_i} 
\end{align}
and let 
\begin{align}\label{ag}
    \sca_g=\{u \in H^1_{u_0}(\om;\R^2): \ \det \nabla u = g \ \textrm{a.e. in } \om\}.
\end{align}

In view of Lemma \ref{MCexample1} and Theorem \ref{minFgivesminD}, in order to find a minimizer of the Dirichlet energy $\D(u)$ on the constrained class of admissible maps $\sca_g$, it is sufficient to minimize $F(u)$ on the larger class $H^1_{u_0}(\om;\R^2)$ whilst ensuring that the minimizer $u$ also belongs to $\sca_g$. Hence we begin by solving a version of \eqref{ELinhombc} tailored to the current setting, namely
\begin{align}\label{ELcurrent}
    \int_{\om} 2\nabla u \cdot \nabla \varphi + \left(\sum_{i=1}^3 f_i \chi_{\omega_i}\right)\cof \nabla u \cdot \nabla \varphi \dx = 0 \quad \forall \varphi \in C_c^{\infty}(\om,\R^2).
\end{align}

\begin{proposition}\label{solveELcurrent} Let $u_0$ be continuous and such that its gradients are given by \eqref{clipper} and let $u \in H^1_{u_0}(\om;\R^2)$ solve \eqref{ELcurrent}. Let 
\begin{align*} \Gamma_{12}=\partial \omega_1 \cap \partial \omega_2 \end{align*}
and similarly for $\Gamma_{23}.$ Then provided the normal derivatives of $u$ exist along $\Gamma_{12}$ and $\Gamma_{23}$, it must be that
\begin{itemize}\item[(i)] $u$ is harmonic in each subdomain $\omega_i$ and
\item[(ii)] the jump conditions
\begin{align}\label{jbc1} (f_2-f_1) J \partial_2 u & =2 \partial_1 u\arrowvert_{\omega_2} - 2 \partial_1 u \arrowvert_{\omega_1} \quad \textrm{along} \ \Gamma_{12} \\
\label{jbc2} (f_3-f_2)J \partial_2 u & =2 \partial_1 u \arrowvert_{\omega_2} - 2 \partial_1 u \arrowvert_{\omega_3} \quad \textrm{along} \ \Gamma_{23} 
\end{align}
hold.  Conversely, $u$ satisfying the conditions in (i) and (ii) must obey \eqref{ELcurrent}. 
\end{itemize}
\end{proposition}
\begin{proof} Since $u \chi_{\omega_i}$ belongs to $H^1(\omega_i;\R^2)$ for any $i$, Piola's identity $\Div \cof \nabla u=0$ shows that $\cof \nabla u \cdot \nabla \varphi$ is a null Lagrangian, and that for a general $\varphi \in C_c^{\infty}(\omega_i;\R^2)$, 
\begin{align*}\int_{\omega_i} \cof \nabla u \cdot \nabla \varphi \dx = \int_{\partial \omega_i} \varphi \cdot \cof \nabla u \, \nu \, \dHone.
\end{align*}
A short calculation shows that 
\begin{align}\label{region1} \int_{\partial \omega_1} \varphi \cdot \cof \nabla u \, \nu \, \dHone & = \int_{\Gamma_{12}} \varphi \cdot J^T\partial_2 u \dHone, \\
\label{region2} \int_{\partial \omega_2} \varphi \cdot \cof \nabla u \, \nu \, \dHone & = \int_{\Gamma_{12}}  \varphi \cdot J \partial_2 u \dHone + \int_{\Gamma_{23}}  \varphi \cdot J^T \partial_2 u \dHone, \\
\label{region3} \int_{\partial \omega_3} \varphi \cdot \cof \nabla u \, \nu \, \dHone & = \int_{\Gamma_{23}} \varphi \cdot J \partial_2 u \dHone. 
\end{align}
Returning to \eqref{ELcurrent}, using \eqref{region1}-\eqref{region3}, and bearing in mind that $f$ is constant on each $\omega_i$, we have 
\begin{align}\label{PDEboundaryform} \int_{\om} 2\nabla u \cdot \nabla \varphi + \left(\sum_{i=1}^3 f_i \chi_{\omega_i}\right) \, \cof \nabla u \cdot \nabla \varphi \dx & = \int_{\om} 2\nabla u \cdot \nabla \varphi \dx + \int_{\Gamma_{12}} (f_2-f_1)  \varphi \cdot J \partial_2 u \dHone   \\  \nonumber & +\int_{\Gamma_{23}} (f_3-f_2) \varphi \cdot J \partial_2 u \dHone. 
\end{align} Now assume that \eqref{ELcurrent} holds. 
Then by taking $\varphi$ in $C_c^{\infty}(\omega_i;\R^2)$ for each $i$ and using \eqref{PDEboundaryform}, it is clear that $u$ is harmonic in each subdomain.   Hence part (i). 
Part (ii) is then a straightforward application of the divergence theorem to \eqref{PDEboundaryform}.
\end{proof}

To conclude this example, we show that the matrices $A_1,A_2$ and $A_3$ can be chosen so that $u=u_0$ both solves \eqref{ELcurrent} and $u_0 \in \sca_g$.  Indeed, it is obvious that $u_0$ obeys part (i) of Proposition \ref{solveELcurrent} and $u_0 \in \sca_g$ where $g$ and $\sca_g$ are given by \eqref{g} and \eqref{ag} respectively.  All that remains to verify are \eqref{jbc1} and \eqref{jbc2}.  

To that end, let $e_1$ and $e_2$ be the canonical basis vectors in $\R^2$.   Firstly, since $u$ is by assumption continuous, Hadamard's condition implies that we must have 
$$ A_1 e_2 = A_2 e_2 = A_3 e_2.$$
Thus the second columns of all the $A_i$ are equal to some $\xi \in \R^2$, say.  To satisfy \eqref{jbc1} and \eqref{jbc2}, 
\begin{align}\label{JJC1}
(f_2-f_1) J \xi & = 2A_2 e_1 - 2 A_1 e_1 \\ \label{JJC2}
(f_3-f_2) J \xi & = 2 A_3 e_1 - 2A_2 e_1
\end{align}
should hold.  Let $\eta= A_2 e_1$.  Then from \eqref{JJC1} and \eqref{JJC2},
\begin{align*}
A_1 e_1  & = 2 \eta - (f_2-f_1)J\xi \\
A_3 e_1 & = 2 \eta + (f_3-f_2)J\xi, 
\end{align*}
and hence suitable matrices are 
\begin{align*} 
A_1 & = (2 \eta - (f_2-f_1)J\xi ) \otimes e_1 + \xi \otimes e_2 \\
A_2& = \eta \otimes e_1 + \xi \otimes e_2 \\
A_3 & =  (2 \eta + (f_3-f_2)J\xi) \otimes e_1 + \xi \otimes e_2
\end{align*} 
where $\xi,\eta \in \R^2$ are free.  We conclude by Theorem \ref{minFgivesminD} that $u_0$ is the global minimizer of $\D(u)$ in the constrained class $\sca_g$.


\begin{remark}\emph{Notice that in this case the minmizer $u=u_0$ behaves as if each Dirichlet energy $\D(v;\omega_i):=\int_{\omega_i}|\nabla v|^2 \dx$ is minimized subject to affine boundary conditions on each $\partial \omega_i$ for $i=1,2,3$. The subtlety here is that we cannot for each $v\in \sca_g$ and each $i=1,2,3$ exploit the quasiconvexity 
\begin{align}\label{ayeayecaptain}
    \D(v;\omega_i) \geq \D(a_i,\omega_i) \quad \forall v\in H^1_{a_i}(\omega_i;\R^2),
\end{align}
in which we employ the notation $a_i:=u_0 \chi_{\omega_i}$, and then simply add the inequalities.  The reason is that there is no guarantee that a typical $v\in \sca_g$ will be affine along $\Gamma_{12}$ or $\Gamma_{23}$, so we do not necessarily have inequality \eqref{ayeayecaptain} for each $i$. This is easily seen:  when the requirement that $v=a_i$ on $\partial \omega_i$ is dropped, it is possible to construct a piecewise affine map $\tilde{v} \in \sca_g$, say, such that for at least one (but not more than two) of the regions $\omega_i$  
$$ \D(a_i;\omega_i) > \D(\tilde{v};\omega_i).$$}
\end{remark}

\subsection{Example 2: a point-contact pressure distribution}\label{sudoku_f}

In this example, we take $\om$ to be the square $Q:=[-1,1]^2$ and assume that it has been divided into the quadrant subsquares $Q_1, Q_2, Q_3, Q_4$ specified in Figure \ref{pic:distribution_Q}.
\begin{figure}[H]
\centering
\begin{minipage}{0.39\textwidth}
\centering
\begin{tikzpicture}[scale=2.5]
\node (A) at (-2,-2) {}; 
\node[right=2 of A.center] (B) {};
\node[right=2 of B.center] (C) {};
\node[right=2 of C.center] (D) {};
\node[above=2 of A.center] (A') {};
\node[above=2 of B.center] (B') {};
\node[above=2 of C.center] (C') {};
\node[above=2 of A'.center] (A'') {};
\node[above=2 of B'.center] (B'') {};
\node[above=2 of C'.center] (C'') {};

\filldraw[thick, top color=gray!30!,bottom color=gray!30!] (B'.center) rectangle node{$Q_1$} (C''.center);

\filldraw[thick, top color=white,bottom color=white] (A'.center) rectangle node{$Q_2$} (B''.center);

\filldraw[thick, top color=gray!30!,bottom color=gray!30!] (A.center) rectangle node{$Q_3$} (B'.center);

\filldraw[thick, top color=white,bottom color=white] (B.center) rectangle node{$Q_4$} (C'.center);

\end{tikzpicture}
\end{minipage}
\begin{minipage}{0.39\textwidth}
\begin{align*}
    Q_1&= [0,1] \times [0,1], \\
    Q_2 & = [-1,0] \times [0,1], \\
    Q_3 & = [-1,0] \times [-1,0], \\
    Q_4 & = [0,1] \times [-1,0].
\end{align*}
\end{minipage}
\caption{Distribution of subdomains $Q_1,\ldots,Q_4$ in Example 2.}
\label{pic:distribution_Q}
\end{figure}
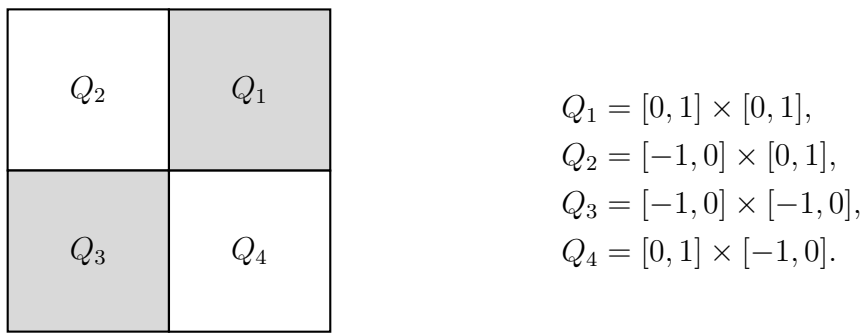
We let $f$ be a pressure function of the form
\begin{align}\label{f2} f = \sum_{i=1}^4 f_i\chi_{_{Q_i}}.\end{align} 
In order to apply Theorem \ref{minFgivesminD}, we must show that $F$ is mean coercive and that the Euler-Lagrange equation \eqref{pdeprime} for $F$, which are given below, can be solved.  In this case, the mean coercivity depends on the values of the constants $f_1,\ldots,f_4$ which, in turn, are entwined with the details of the solution $u$. See \eqref{concrete} for a particular instance of a pressure function $f$ that `fits' with the solution.  Accordingly, we postpone to Lemma \ref{gap} the argument needed to prove mean coercivity and begin by seeking a solution $u$ to the Euler-Lagrange equation for $F(u)$, namely
\begin{align}\label{pdeprime} \int_{Q} 2 \nabla u \cdot \nabla \varphi +  \left(\sum_{i=1}^4 f_i\chi_{_{Q_i}}\right) \, \cof \nabla u \cdot \nabla \varphi \dx = 0 \quad \quad \varphi \in C_c^{\infty}(Q;\R^2).
\end{align}
 In the following statement the sets $\Gamma_{ij}$ are defined in the same way as in Proposition \ref{solveELcurrent}, so that $\Gamma_{12}:=\partial Q_1 \cap \partial Q_2$, and so on.
\begin{proposition}\label{nec2} Let $u \in H^1(Q;\R^2)$ solve \eqref{pdeprime}.  Then provided the normal derivatives of $u$ exist along $\Gamma_{12},\ldots,\Gamma_{41}$, it must be that 
\begin{itemize}\item[(i)] $u$ is harmonic in each subdomain $Q_i$, and
\item[(ii)] the jump conditions 
\begin{align}\label{jc1} 2\partial_2 u\arrowvert_{Q_4} - 2 \partial_2 u\arrowvert_{Q_1} +(f_4-f_1) J \partial_1 u  & =0 \quad \textrm{along} \ \Gamma_{14}\\
\label{jc2}  2\partial_1 u\arrowvert_{Q_2} - 2\partial_1 u\arrowvert_{Q_1} + (f_1-f_2)J\partial_2 u &  = 0 \quad \textrm{along} \ \Gamma_{12}  \\
\label{jc3} 2\partial_2 u\arrowvert_{Q_3}- 2\partial_2 u\arrowvert_{Q_2} +(f_3-f_2) J \partial_1 u & = 0 \quad \textrm{along} \ \Gamma_{23} \\
\label{jc4}
2\partial_1 u\arrowvert_{Q_3} - 2\partial_1 u\arrowvert_{Q_4}
+(f_4-f_3) J \partial_2 u & = 0 \quad \textrm{along} \ \Gamma_{34}.
\end{align}
\noindent hold.
Conversely, $u$ satisfying the conditions in (i) and (ii) must obey \eqref{pdeprime}. 
\end{itemize}
\end{proposition}
\begin{proof} This is so similar to the proof of Proposition \eqref{solveELcurrent} that we omit it.
  \end{proof}

The next result demonstrates, by brute force, that solutions to the system \eqref{pdeprime} exist.   

\begin{proposition}\label{hydrostatic_solutions} Let $f$ given by \eqref{f2} and  define, in complex coordinates, the function $u(z)$ by
\begin{align}\label{fullsol} u(z) = \sum_{n \ \textrm{even}} a_n (z^{n} - \bar{z}^{n}) + u^{(1)}(z)\chi_{_{Q_1}}+ u^{(2)}(z)\chi_{_{Q_2}}+u^{(3)}(z)\chi_{_{Q_3}}+u^{(4)}(z)\chi_{_{Q_4}},\end{align} where  
\begin{align*}
    u^{(1)}(z)  
    & = \sum_{n \ \textrm{odd}} \lambda_n^{(1)}\left( 2\beta\gamma(z^n+\bar{z}^n) -p(z^n-\bar{z}^n)\right), \\
    u^{(2)}(z)  
    & = \sum_{n \ \textrm{odd}} \lambda_n^{(2)}\left( 2\gamma\delta(z^n+\bar{z}^n) -p(z^n-\bar{z}^n)\right),\\
    u^{(3)}(z) &=
    \sum_{n \ \textrm{odd}} \lambda_n^{(3)}\left( 2\gamma\delta(z^n+\bar{z}^n) +(4\beta-4\delta)(z^n-\bar{z}^n)\right),\\
    u^{(4)}(z) &=
    \sum_{n \ \textrm{odd}} \lambda_n^{(4)}\left( 2\beta\gamma(z^n+\bar{z}^n) -(4\beta-4\delta)(z^n-\bar{z}^n)\right).
\end{align*}
The constants $a_n$ may be complex, while 
$\alpha:=f_2-f_1$, $\beta:=f_3-f_2$, $\gamma:=f_4-f_3$, $\delta:=f_1-f_4$ and $p:=\beta\gamma\delta+4\beta-4\delta$ are all real. Assume that $\alpha \neq 0,\beta \neq 0, \gamma \neq 0, \delta \neq 0$, and $p \neq 0$.  Then there are two possibilities according to whether the quantity
\begin{align}\label{bruckner3}
    \Delta:= \alpha\beta\gamma\delta+4(\beta-\delta)(\alpha+\gamma)
\end{align}
vanishes or not, which are:
\begin{itemize}\item[($\mathbf{\Delta \neq 0}$)] all $\lambda_n^{(j)}$ must vanish, and the solution \eqref{fullsol} to $\eqref{pdeprime}$ is harmonic on all of $Q$ provided the sequence $(a_n)_{n \ \mathrm{even}}$ is chosen so that $\sum_{n \ \textrm{even}} a_n (z^{n} - \bar{z}^{n})$ converges;
\item[($\mathbf{\Delta = 0}$)] provided for each $j=1,\ldots,4$, the sequences $(\lambda_n^{(j)})_{n \ \textrm{odd}}$
are chosen to ensure that the corresponding series $u^{(j)}(z)$ converges, any function of the form $u(z)$ given by \eqref{fullsol} is a solution to \eqref{pdeprime}.
\end{itemize}
   
\end{proposition}
\begin{proof} We assume that the conditions set out in (i) and (ii) of Proposition \ref{nec2} apply.  In terms of complex coordinates, where $(x_1,x_2)=:(x,y)$ is identified with $z:=x+iy$ and $u(x_1,x_2)$ is identified with $u(z):=u_1(x,y)+iu_2(x,y)$, let $u^{(k)}$ be the restriction of $u$ to the quadrant $Q_k$.  Since, for each $k=1,\ldots,4$,  $u^{(k)}$ is harmonic in $Q_k$, a standard representation theorem implies that
\begin{align}\label{holorep} u^{(k)}(z)= F^{(k)}(z) + \overline{G^{(k)}(z)},\end{align}
where $F^{(k)}$ and $G^{(k)}$ are functions holomorphic in $Q_k$.  

Turning to (ii), we begin by converting equations \eqref{jc1}---\eqref{jc4} into complex form, and then make use of the relationships $F^{(k)}_y=iF^{(k)}_x$ and $G^{(k)}_y=iG^{(k)}_x$, which hold at any point in $Q_k$ and which we further assume to hold in an appropriate limiting sense at points along the coordinate axes bordering $Q_k$.     For all $m,n$, let $\eps_{mn}=f_m-f_n$.   Then \eqref{jc1} becomes 
\begin{align}\label{pinky} F^{(4)}_x -\overline{G^{(4)}_x} & = F^{(1)}_x -\overline{G^{(1)}_x} +\frac{\eps_{14}}{2}\left( F^{(1)}_x +\overline{G^{(1)}_x}\right) \quad\quad z=x+0i \in \Gamma_{14}.
\end{align}
Applying the assumption that $u$ is continuous across $\Gamma_{14}$, so that $u^{(4)}(x+0i)=u^{(1)}(x+0i)$ for $0\leq x \leq 1$, and by further assuming that we may differentiate this expression, we obtain (bearing \eqref{holorep} in mind)
\begin{align}\label{perky}
F^{(4)}_x +\overline{G^{(4)}_x} & = F^{(1)}_x +\overline{G^{(1)}_x} \quad\quad z=x+0i \in \Gamma_{14}.
\end{align}
We regard the functions $F^{(1)}$ and $G^{(1)}$ as being `free', and solve \eqref{pinky} and \eqref{perky} for $F^{(4)}_x$ and $\overline{G^{(4)}_x} $, giving, for  $z=x+0i \in \Gamma_{14}$, 
\begin{align}\label{14first} F^{(4)}_x  & = \left(1+\frac{\eps_{14}}{4}\right) F^{(1)}_x + \frac{\eps_{14}}{4} \overline{G^{(1)}_x} \\
\label{14second} \overline{G^{(4)}_x}  & =    \left(1-\frac{\eps_{14}}{4}\right)\overline{G^{(1)}_x}  - \frac{\eps_{14}}{4}  F^{(1)}_x.
\end{align}

Doing likewise with \eqref{jc2}, taking care in this case to replace, when $k=1,2$, the normal derivatives $\partial_xF^{(k)}$ and 
$\overline{\partial_xG^{(k)}}$ by the tangential derivatives $-i \partial_yF^{(k)}$ and $i \overline{\partial_yG^{(k)}}$ respectively, we obtain, when $z=0+yi \in \Gamma_{12}$, the equations
\begin{align}
\label{21first} F^{(2)}_y & = \left(1-\frac{\eps_{21}}{4}\right) F^{(1)}_y - \frac{\eps_{21}}{4} \overline{G^{(1)}_y} \\
\label{21second} \overline{G^{(2)}_y} & = \left(1+\frac{\eps_{21}}{4}\right) \overline{G^{(1)}_y}  + \frac{\eps_{21}}{4}F^{(1)}_y.
\end{align}

Similarly, from \eqref{jc3}, we find that
\begin{align}\label{32first} F^{(3)}_x  & = \left(1+\frac{\eps_{32}}{4}\right) F^{(2)}_x + \frac{\eps_{32}}{4} \overline{G^{(2)}_x} \\
\label{32second} \overline{G^{(3)}_x}  & =    \left(1-\frac{\eps_{32}}{4}\right)\overline{G^{(2)}_x}  - \frac{\eps_{32}}{4}  F^{(2)}_x
\end{align}
for $z=x+0i \in \Gamma_{23}$,  while \eqref{jc4} yields
\begin{align}
\label{43first} F^{(4)}_y & = \left(1-\frac{\eps_{43}}{4}\right) F^{(3)}_y - \frac{\eps_{43}}{4} \overline{G^{(3)}_y} \\
\label{43second}\overline{G^{(4)}_y} & = \left(1+\frac{\eps_{43}}{4}\right) \overline{G^{(3)}_y}  + \frac{\eps_{43}}{4}F^{(3)}_y
\end{align}
if $z=0+yi \in \Gamma_{34}$. 

To solve this system, we suppose for now that $F^{(k)}$ and $G^{(k)}$ can be written as formal power series, thus:
\begin{align}\label{PS1}
F^{(k)}(z) & = \sum_{n=0}^{\infty} a_{n}^{(k)} z^n \\ 
\label{PS2}G^{(k)}(z) & = \sum_{n=0}^{\infty} b_{n}^{(k)} z^n
\end{align} 
for $k=1,\ldots,4$.  The introduction of \eqref{PS1} and \eqref{PS2} allows us to relate the various derivatives $F_x^{(k)}$ and $F_y^{(k)}$ (and similarly $G^{(k)}_x$ and $G^{(k)}_y$) appearing in \eqref{14first}---\eqref{43second}, and so `close' the system, as follows.  

Substituting \eqref{PS1} and \eqref{PS2} into \eqref{14first} gives
$$ \sum_{n=1}^{\infty} n a_n^{(4)} x^{n-1} = \sum_{n=1}^{\infty} \left(1+\frac{\eps_{14}}{4}\right) n a_n^{(1)} x^{n-1} + \frac{\eps_{14}}{4} n \overline{b_n^{(1)}} x^{n-1} \quad \quad 0 \leq x \leq 1,$$
which is satisfied if 
\begin{align}\label{A} a_n^{(4)} & =  \left(1+\frac{\eps_{14}}{4}\right) a_n^{(1)} + \frac{\eps_{14}}{4}  \overline{b_n^{(1)}} \quad \quad n \geq 1,
\end{align}
while \eqref{14second} holds if 
\begin{align}\label{B} \overline{b_n^{(4)}} & =  \left(1-\frac{\eps_{14}}{4}\right)\overline{ b_n^{(1)}} - \frac{\eps_{14}}{4} a_n^{(1)} \quad \quad n \geq 1.
\end{align}

Now consider \eqref{21first} and \eqref{21second}, both of which require tangential derivatives $G^{(k)}_y$ for $k=1,2$.  From \eqref{PS2}, we have
\begin{align*} G^{(1)}_y(z) = \sum_{n=1}^{\infty} i n b^{(1)}_n z^{n-1}, 
\end{align*}
and hence, with $z=0+iy$ along $\Gamma_{12}$,
\begin{align*} \overline{G^{(1)}_y(iy)} & = \sum_{n=1}^{\infty} \overline{i n b^{(1)}_n (iy)^{n-1}} \\
& =   \sum_{n=1}^{\infty} n \bar{i}^n  \overline{b^{(1)}_n} y^{n-1}.
\end{align*}
Hence, \eqref{21first} gives 
\begin{align*}
\sum_{n=1}^{\infty} n i^n a^{(2)}_n y^{n-1} & = \sum_{n=1}^{\infty} n i^n\left(1-\frac{\eps_{21}}{4}\right) a^{(1)}_n y^{n-1} - n \bar{i}^n\frac{\eps_{21}}{4} \overline{b^{(1)}_n} y^{n-1}, 
\end{align*}
which is equivalent to
\begin{align}\label{D} a_n^{(2)} = \left(1-\frac{\eps_{21}}{4}\right) a^{(1)}_n - (-1)^n \frac{\eps_{21}}{4} \overline{b^{(1)}_n}  \quad \quad n \geq 1.
\end{align} 
Equation \eqref{21second} leads to 
\begin{align}\label{C} \overline{b^{(2)}_n} & = \left(1+\frac{\eps_{21}}{4}\right)\overline{b^{(1)}_n} + (-1)^n \frac{\eps_{21}}{4} a^{(1)}_n  \quad \quad n \geq 1.
\end{align}
Proceeding similarly with the remaining equations leads to 
\begin{align}\label{Ebygum}  a^{(3)}_n & =  \left(1+\frac{\eps_{32}}{4}\right) a^{(2)}_n + \frac{\eps_{32}}{4} \overline{b^{(2)}_n} \\
\label{F} \overline{b^{(3)}_n} & =  \left(1- \frac{\eps_{32}}{4}\right) \overline{b^{(2)}_n} - \frac{\eps_{32}}{4} a^{(2)}_n \\
\label{H} a^{(4)}_n & = \left(1- \frac{\eps_{43}}{4}\right) a^{(3)}_n - (-1)^n \frac{\eps_{43}}{4}\overline{b^{(3)}_n} \\
\label{G} \overline{b^{(4)}_n} & = \left(1+ \frac{\eps_{43}}{4}\right) \overline{b^{(3)}_n} +  (-1)^n \frac{\eps_{43}}{4} a^{(3)}_n. 
\end{align}

Let $\mu=e_1+e_2$, $\nu=e_2-e_1$, and define for all real $w$ the matrices 
\begin{align}
   \label{matrixA} A(w):= \1 + w \, \nu \otimes \mu.
\end{align}
For each $n \geq 1$ and $j=1\ldots 4$ define vectors $v_n^{(j)}$ by
\begin{align*}v^{(j)}_n: & =
\left(\begin{array}{c} a^{(j)}_n \\ \overline{b^{(j)}_n} \end{array}\right).
\end{align*}
Then \eqref{A}-\eqref{G} can be written as  
    \begin{align}
\label{21} v_n^{(2)} = \left\{\begin{array}{l l} A(\eps_{21}/4) v_n^{(1)} & n \ \textrm{even} \\
A^T(\eps_{21}/4) v_n^{(1)} & n \ \textrm{odd},\end{array}\right.
\end{align}
\begin{align}
\label{32} v_n^{(3)} = A(-\eps_{32}/4) v_n^{(2)},
\end{align}
\begin{align}
\label{43} v_n^{(4)} = \left\{\begin{array}{l l} A(\eps_{43}/4) v_n^{(3)} & n \ \textrm{even} \\
A^T(\eps_{43}/4) v_n^{(3)} & n \ \textrm{odd},\end{array}\right.
\end{align}
\begin{align}
\label{14} v_n^{(1)} = A(\eps_{14}/4) v_n^{(4)}.
\end{align}
\noindent{\textbf{Case (i)}.} When $n$ is even the system has a solution only if $v_n^{(1)}$ obeys
\begin{align}\label{evp1}
    A(\eps_{14}/4)\,A(\eps_{43}/4)\,A(-\eps_{32}/4)\,A(\eps_{21}/4) v_n^{(1)} = v_n^{(1)}
\end{align}
Since $A(r)\,A(s)=A(r+s)$ for any real $r, s$, it follows that $A(\zeta)v_n^{(1)}=v_n^{(1)}$ with $\zeta=(\eps_{14}+\eps_{43}-\eps_{32}+\eps_{21})/4=(f_2-f_3)/2$.  Since $f_3 \neq f_2$ by hypothesis, we may assume that $\zeta \neq 0$ and hence, by \eqref{matrixA}, $A(\zeta)v_n^{(1)}=v_n^{(1)}$ only if $v_n^{(1)}$ lies in the kernel of $\nu \otimes \mu$.  It then follows from \eqref{21}-\eqref{43} that $v_n^{(4)}=v_n^{(3)}=v_n^{(2)}=v_n^{(1)}$ for any even $n$, each vector being proportional to $\nu$.  Recalling that $u^{(j)}(z)$ is the restriction of $u$ solving \eqref{pdeprime} to $Q_j$, we see, for instance, that 
\begin{align*}
    u^{(1)}(z) & = F^{(1)}(z)+\overline{G^{(1)}(z)} \\ & = \sum_{n=1}^{\infty} a_n^{(1)}z^n + \overline{b_{n}^{(1)}}\bar{z}^n
\end{align*}
is such that its `even part'
\begin{align*}
u^{(1)}_{\textrm{even}}(z)& :=\sum_{n \ \textrm{even}} a_n^{(1)}z^n + \overline{b_{n}^{(1)}}\bar{z}^n \\
& = \sum_{n \ \textrm{even}} a_n^{(1)}(z^n-\bar{z}^n)
\end{align*}
agrees with $u^{(2)}_{\textrm{even}}(z)$, $u^{(3)}_{\textrm{even}}(z)$ and $u^{(4)}_{\textrm{even}}(z)$.  Thus the even part of $u$, which we can write as
\begin{align}u_{\textrm{even}}(z):=u^{(1)}_{\textrm{even}}(z)\,\chi_{_1}(z) + \ldots +u^{(4)}_{\textrm{even}}(z)\, \chi_{_4}(z)
\end{align}
is harmonic on the whole domain.  Here, $\chi_{_1}:=\chi_{_{Q_1}}$, and so on.
\\
\noindent{\textbf{Case (ii)}.}  
When $n$ is odd the system has a solution only if $v_n^{(1)}$ obeys
\begin{align}\label{matrixC}
    A(\eps_{14}/4)\,A^T(\eps_{43}/4)\,A(-\eps_{32}/4)\,A^T(\eps_{21}/4) v_n^{(1)} = v_n^{(1)}.
\end{align}
Denoting by $C$ the matrix appearing in the left-hand side of \eqref{matrixC}, and letting $\alpha=\eps_{21}$, $\beta=\eps_{32}$, $\gamma=\eps_{43}$ and $\delta=\eps_{14}$,
we find (in our case, using Maple$^{\mathrm{TM}}$) that $\det(C-\1)=0$ only if 
\begin{align}
  \label{sing}  \alpha\beta\gamma\delta+4(\beta-\delta)(\alpha+\gamma) = 0,
\end{align}
which we recognize as the condition $\Delta=0$.
Solving for $\alpha$ leads to  
\begin{align}
\label{aa}\alpha = -\frac{4\gamma(\beta - \delta)}{\beta\gamma\delta + 4\beta - 4\delta}
    \end{align}
as long as 
\begin{align}\label{p3}
    p=\beta\gamma\delta + 4\beta - 4\delta
\end{align}
obeys $p \neq0$.  Assuming that $p\neq 0$ and choosing $\alpha$ as in \eqref{aa}, the equation $(C-\1)v_n^{(1)}=0$ is solved by any multiple of
\begin{align}\label{vn1}
 v^{(1)}_n  =\left(\begin{array}{c} 2\beta\gamma-p  \\ p+2\beta\gamma \end{array}\right).
\end{align}
Using \eqref{32} leads to 
\begin{align}\label{vn2}
    v_n^{(2)} = \left(\begin{array}{c} 2\gamma\delta-p  \\ p+2\gamma\delta \end{array}\right),
\end{align}
which, through \eqref{43}, yields
\begin{align}\label{vn3}
    v_n^{(3)} = \left(\begin{array}{c} 2\gamma\delta+4\delta-4\beta  \\ 2\gamma\delta + 4\beta-4\delta \end{array}\right).
\end{align}
Finally, \eqref{14} gives 
\begin{align}\label{vn4}
    v_n^{(4)} = \left(\begin{array}{c} 2\beta\gamma+4\delta-4\beta  \\ 2\beta\gamma + 4\beta  -4\delta \end{array}\right).
\end{align}
In addition to the standing assumptions that $\alpha \neq 0$, $\beta \neq 0$, $\gamma \neq 0$, $\delta \neq 0$, and $p \neq 0$, the only condition needed to ensure that the vectors $v_n^{(j)}$ are distinct is $\beta \neq \delta$, which, in terms of the original variables, amounts to $f_1+f_2 \neq f_3+f_4$.  In summary, using $v_n^{(1)}$ defined in \eqref{vn1}, the `odd part' of the solution $u^{(1)}(z)$ to \eqref{pdeprime} can now formally be written as 
\begin{align*}
    u^{(1)}_{\textrm{odd}}(z) & =\sum_{n \ \textrm{odd}}a_n^{(1)}z^n + \overline{b_n^{(1)}}\bar{z}^n \\
    & = \sum_{n \ \textrm{odd}} 2\beta\gamma(z^n+\bar{z}^n) -p(z^n-\bar{z}^n).
    \end{align*}
    Similarly, each $u^{(j)}_{\textrm{odd}}(z)$ for $j=2,3,4$ can be constructed using the components of $v_n^{(j)}$ as given by \eqref{vn2}-\eqref{vn4}. 
   \end{proof}
\begin{remark} \emph{We remark that the difference $f_2-f_3$ is singled out as a consequence of the choice we made just after \eqref{perky} to regard the functions $F^{(1)}$ and $G^{(1)}$ as being `free', leading to the eigenvalue problems \eqref{evp1} and \eqref{matrixC} and associated eigenvector $v_n^{(1)}$.  It seems that other free variable choices lead to the same dependence on quantities of the form $f_i-f_{i+1}$, with subscripts calculated modulo 5, and that by a rotation of the initial frame, all these solutions are equivalent.  In particular, there should be nothing special about $f_2-f_3$ apart from its being a difference of the values taken by $f$ on neighbouring subdomains of $Q$.} \end{remark}

We now apply Proposition \ref{hydrostatic_solutions} to the pressure function 
\begin{align}\label{concrete} f = \sigma\chi_{Q_{2}}-(\tau+\sigma)\chi_{Q_{1}}+(\tau-\sigma)\chi_{Q_{3}}+\sigma \chi_{Q_{4}},
\end{align}
where $\tau$ and $\sigma$ are parameters chosen in Proposition \ref{step1} below so that a solution $u$ to \eqref{pdeprime} exists.  Subsequently, via Lemma \ref{gap} and Proposition \ref{step2}, we tune $\sigma$ and $\tau$ in order that $F(u)$ is mean coercive.

\begin{proposition}\label{step1} Let $f$ be given by \eqref{concrete}.  Then coefficients $\sigma$ and $\tau$ can be chosen so that $\Delta$ defined by \eqref{bruckner3} obeys $\Delta=0$, $p$ defined by \eqref{p3} obeys $p\neq 0$, and all other assumptions concerning the quantities $\alpha, \beta, \gamma$ and $\delta$ defined in the statement of Proposition \ref{hydrostatic_solutions} are satisfied.  In particular, modulo the addition of a function harmonic on $Q$, solutions to \eqref{pdeprime} can be expressed as weighted sums of the functions $u_n(z)$ given by \eqref{bb317}. 
\end{proposition}
\begin{proof}
    Using the definitions of $\alpha,\beta,\gamma, \delta$ given in Proposition \ref{hydrostatic_solutions}, we find that with $f$ as in \eqref{concrete}, 
\begin{align}
\label{alpha} \alpha & =2\sigma+\tau,    \\
 \label{beta}\beta &  = \tau - 2\sigma, \\
 \label{gamma}\gamma & = -\beta, \\
 \label{delta}\delta & = -\alpha.
\end{align}
According to Proposition \ref{hydrostatic_solutions}(b), non-smooth solutions to \eqref{pdeprime} exist provided:
\begin{itemize}\item[(a)] $\Delta=0$ where, in this case, $\Delta=\alpha^2\beta^2+4(\alpha^2-\beta^2)$;
\item[(b)] $p \neq 0$, where $p=\beta^2\alpha +4(\beta+\alpha)$;
\item[(c)] $\alpha \neq 0$, $\beta \neq 0$, and 
\item[(d)] $\beta +\alpha \neq 0$.
\end{itemize}
Hence, from (a), we set 
\begin{align} \label{alphabet} \alpha & =\frac{2\beta}{\sqrt{\beta^2+4}}\end{align} and we find that (b)-(d) are satisfied as long as $\beta \neq 0$ and $\tau \neq0$.

The recipe of Proposition \ref{hydrostatic_solutions} now ensures that solutions to \eqref{pdeprime} are, up to the addition of a function that is harmonic everywhere in $Q$ as described in Proposition \ref{hydrostatic_solutions}(a), and still in complex notation, weighted sums of the `building block' functions

\begin{align}
    u_n(z)&=\left\{\begin{array}{l l} -2\beta^2 (z^n+\bar{z}^n)- 2\beta(2+\sqrt{\beta^2+4}) (z^n-\bar{z}^n) & z \in Q_1 \\ 
    \frac{4\beta^2}{\sqrt{\beta^+4}} (z^n+\bar{z}^n)- 2\beta(2+\sqrt{\beta^2+4}) (z^n-\bar{z}^n) & z \in Q_2 \\
    \frac{4\beta^2}{\sqrt{\beta^+4}} (z^n+\bar{z}^n)- 4\beta\left(1+\frac{2}{\sqrt{\beta^2+4}}\right) (z^n-\bar{z}^n) & z \in Q_3 \\
    -2\beta^2 (z^n+\bar{z}^n)- 4\beta\left(1+\frac{2}{\sqrt{\beta^2+4}}\right) (z^n-\bar{z}^n) & z \in Q_4,\end{array}\right. \label{bb317}
\end{align}
where $n$ is an odd natural number.  \end{proof}

We remark that by setting $z = R e^{i\theta}$, we have 
\begin{align*}
      z^n+\bar{z}^n & = 2R^n \cos(n \theta) \\
    z^n-\bar{z}^n & = 2R^n \sin(n \theta)i,
    \end{align*}
and so $\R^2$-valued `building block' functions are, in plane polar coordinates $(R,\theta)$,
\begin{align*}
    u_n(R,\theta)=\left\{\begin{array}{l l} R^n D_1 e(n\theta) & (R,\theta) \in Q_1 \\
    R^n D_2 e(n\theta) & (R,\theta) \in Q_2 \\
    R^n D_3 e(n\theta) & (R,\theta) \in Q_3 \\
    R^n D_4 e(n\theta) & (R,\theta) \in Q_4,
    \end{array}\right.
\end{align*}
where $e(n\theta)=(\cos(n\theta),\sin(n\theta))^T$ and $D_1,\ldots,D_4$ are the diagonal matrices given by 
\begin{align*}
    D_1 & = \mathrm{diag\left(-4\beta^2,-4\beta(2+\sqrt{4+\beta^2}\right)},\\
    D_2 & = \mathrm{diag\left(\frac{8\beta^2}{\sqrt{\beta^+4}},-4\beta(2+\sqrt{4+\beta^2}\right)}, \\
    D_3 & = \mathrm{diag\left(\frac{8\beta^2}{\sqrt{\beta^+4}},-8\beta\left(1+\frac{2}{\sqrt{4+\beta^2}}\right)\right)} ,\\
    D_4 & = \mathrm{diag\left(-4\beta^2,-8\beta\left(1+\frac{2}{\sqrt{4+\beta^2}}\right)\right)}.
\end{align*}
As a particular example, note that when $n=1$, the solutions $u_n$ are just piecewise affine functions given in Cartesian coordinates by 
\begin{align*}
    u_1(x_1,x_2) = D_k \left(\begin{array}{c}x_1 \\ x_2 \end{array}\right)\quad \quad \mathrm{if } \left(\begin{array}{c}x_1 \\ x_2 \end{array}\right) \in Q_k
\end{align*}
for $k=1,\ldots,4$.  Since $\mathrm{rank}(D_{k+1}-D_k)=1$ for $k=1,\ldots,4$, it follows immediately that Hadamard's rank-one condition holds, as expected.

Having established conditions under which solutions to \eqref{pdeprime} exist, we now tune $\sigma$ and $\tau$ in order that $F$ is mean coercive. The first step is to rewrite $F(\varphi)$ slightly when $\varphi \in C_c^{\infty}(Q;\R^2)$.  Let $\varphi \in C_c^{\infty}(Q;\R^2)$ and note that, since $\det \nabla \varphi$ is a null Lagrangian, it holds that 
$$\int_Q \det \nabla \varphi \dx=0.$$
In particular, we can subtract $\sigma \int_Q \det \nabla \varphi \dx$ from $F(\varphi)$ without changing its value, which leads to the equivalent form
\begin{align}
    \label{decomp}F(\varphi) & = \int_{Q}|\nabla \varphi|^2 -(\tau +2\sigma)\det \nabla \varphi \, \chi_{Q_1} + (\tau - 2\sigma)\det \nabla \varphi \, \chi_{Q_3}  \dx \\ \nonumber
    & = \underbrace{\int_{Q}\lambda|\nabla \varphi|^2 -2\sigma\det \nabla \varphi \, (\chi_{Q_1}+\chi_{Q_3})\dx}_{=:\small{F_{\lambda}\left(\varphi, \  
\scalebox{0.7}{{\begin{tabular}{|c|c|}
 \hline
$0$ &$-2\sigma$ 
\\ \hline
$-2\sigma$ & $0$ \\
\hline
\end{tabular}}}\, \right)}} + \underbrace{\int_{Q}(1-\lambda)|\nabla \varphi|^2 + \tau\det \nabla \varphi \, (\chi_{Q_3}-\chi_{Q_1})  \dx}_{\small{F_{1-\lambda}\left(\varphi,\  
\scalebox{0.7}{{\begin{tabular}{|c|c|}
 \hline
$0$ &$-\tau$ 
\\ \hline
$\tau$ & $0$ \\
\hline
\end{tabular}}}\, \right)}}. 
\end{align}
Here, $\lambda \in (0,1)$ will be chosen shortly and in accordance with the following lemma.

\begin{lemma}\label{gap}
Both of the functionals  $F_{\lambda}\left(\varphi, \  
\scalebox{0.7}{{\begin{tabular}{|c|c|}
 \hline
$0$ &$-2\sigma$ 
\\ \hline
$-2\sigma$ & $0$ \\
\hline
\end{tabular}}}\, \right)$ and $F_{1-\lambda}\left(\varphi,\  
\scalebox{0.7}{{\begin{tabular}{|c|c|}
 \hline
$0$ &$-\tau$ 
\\ \hline
$\tau$ & $0$ \\
\hline
\end{tabular}}}\, \right)$ defined in \eqref{decomp} are nonnegative on $C_c^{\infty}(Q;\R^2)$ provided $\lambda, \sigma$ and $\tau$ obey
\begin{align}\label{tightrope}
\frac{|\sigma|}{2} \leq \lambda \leq 1-\frac{|\tau|}{\sqrt{8}}.
\end{align}
Moreover, if \eqref{tightrope} is strengthened to 
\begin{align}\label{tightrope2} \frac{|\sigma|}{2} < \lambda < 1-\frac{|\tau|}{\sqrt{8}}.
\end{align}
then $F$ is mean coercive.
\end{lemma}
\begin{proof}
Firstly, write
\begin{align*}
  F_{\lambda}\left(\varphi, \  
\scalebox{0.7}{{\begin{tabular}{|c|c|}
 \hline
$0$ &$-2\sigma$ 
\\ \hline
$-2\sigma$ & $0$ \\
\hline
\end{tabular}}}\, \right) & = \lambda \int_{Q} |\nabla \varphi|^2 -\frac{2\sigma}{\lambda}\det \nabla \varphi \, (\chi_{Q_1}+\chi_{Q_3})\dx  
\end{align*}
and let $c=\frac{2\sigma}{\lambda}$.  Let $\omega_n \subset Q$ be any sequence of subsets with the properties that (i) $\omega_n$ is homeomorphic to an open ball in $\R^2$, (ii) $Q_1\cup Q_3 \subset \omega_n$ for all $n$, and (iii) $\chi_{\omega_{n}} \to \chi_{Q_1\cup Q_3}$ as $n \to \infty$.  For example, the sets 
\begin{align*}
    \omega_n:=\left\{x \in Q: \ \textrm{dist}\,(x,Q_1\cup Q_3) <\frac{1}{n}\right\}
\end{align*}
fulfill (i)-(iii).  Then, by \cite[Proposition 3.4]{BKV23}, for each $n \in \N$ the functional 
\begin{align*} F_n(\varphi):=\int_{Q} |\nabla \varphi|^2 - c \det \nabla \varphi \, \chi_{\omega_n} \dx \end{align*}
is nonnegative on $C_c^{\infty}(Q;\R^2)$ if and only if $|c|\leq 4$.  By letting $n \to \infty$ and noting that, by dominated convergence,  
\begin{align*}
    F_n(\varphi) \to \int_{Q} |\nabla \varphi|^2 -c\det \nabla \varphi \, (\chi_{Q_1}+\chi_{Q_3})\dx,  
\end{align*}
it follows that so too is $F_{\lambda}(\varphi)$ nonnegative on 
$C_c^{\infty}(Q;\R^2)$ if and only if $|c|\leq 4$. This is equivalent to $\frac{|\sigma|}{2} \leq \lambda$.  

Next, rewrite 
\begin{align*}
    F_{1-\lambda}\left(\varphi,\  
\scalebox{0.7}{{\begin{tabular}{|c|c|}
 \hline
$0$ &$-\tau$ 
\\ \hline
$\tau$ & $0$ \\
\hline
\end{tabular}}}\, \right) = (1-\lambda)\int_Q |\nabla \varphi|^2 +\frac{\tau}{1-\lambda} \det \nabla \varphi \, (\chi_{Q_3}-\chi_{Q_1})  \dx
\end{align*}
and let $d=\frac{\tau}{1-\lambda}$.  According to 
\cite[Proposition 4.6]{BKV23}, the functional 
\begin{align*}
   \varphi \mapsto \int_Q |\nabla \varphi|^2 +d \det \nabla \varphi \, (\chi_{Q_3}-\chi_{Q_1})  \dx
\end{align*}
is nonnegative on $C_c^{\infty}(Q;\R^2)$ if and only if $|d|\leq \sqrt{8}$, which is equivalent to $\lambda \leq 1-\frac{|\tau|}{\sqrt{8}}$.  Putting both inequalities involving $\lambda$ together yields \eqref{tightrope}. 

Now suppose that \eqref{tightrope2} holds so that, in particular, $|\frac{2\sigma}{\lambda}|<4$. Consider 
\begin{align*}
  F_{\lambda}\left(\varphi, \  
\scalebox{0.7}{{\begin{tabular}{|c|c|}
 \hline
$0$ &$-2\sigma$ 
\\ \hline
$-2\sigma$ & $0$ \\
\hline
\end{tabular}}}\, \right) & = \lambda \int_{Q} |\nabla \varphi|^2 -\frac{2\sigma}{\lambda}\det \nabla \varphi \, (\chi_{Q_1}+\chi_{Q_3})\dx  \\
& = \lambda \eps \int_{Q}|\nabla \varphi|^2 \, \dx + \lambda(1-\eps)\left(\int_{Q}|\nabla \varphi|^2 -\frac{2\sigma}{\lambda(1-\eps)}\det \nabla \varphi \, \dx \right)
\end{align*}
and notice that for sufficiently small $\eps$ we may assume that $|\frac{2\sigma}{\lambda}(1-\eps)|\leq 4$ and hence that the functional on the right is nonnegative.  Using this and \eqref{decomp} we therefore have that 
$$F(\varphi) \geq \gamma \int_{Q} |\nabla \varphi|^2 \,\dx \quad \forall \varphi \in C_c^{\infty}(Q;\R^2)$$
where $\gamma=\lambda \eps$.
\end{proof}

It follows from Lemma \ref{gap} and \eqref{decomp} that as long as  \begin{align}\label{pared}\frac{|\sigma|}{2} \leq 1-\frac{|\tau|}{\sqrt{8}},\end{align} then $\lambda$ can be chosen to lie between these values, and hence $F(\varphi) \geq 0$ for all $\varphi \in C_c^{\infty}(Q;\R^2)$.  But the choice of $\sigma$ and $\tau$ is not free: one parameter is subordinated to the other through \eqref{alpha}, \eqref{beta} and \eqref{alphabet}, which when combined lead to
\begin{align}\label{sigmatau}
\tau^4-8\tau^2\sigma^2+32\tau\sigma+16\sigma^4=0.
\end{align}
Thus, in order to conclude, we seek solutions to \eqref{sigmatau} such that \eqref{pared} holds.  
A brief numerical investigation, which we summarise in Proposition \ref{step2} below, reveals (at least) one `branch' of solutions $(\sigma,\tau)$ which obey both \eqref{sigmatau} and
\begin{align}\label{yy}
  y(\sigma,\tau):=\frac{|\sigma|}{2}  + \frac{|\tau|}{\sqrt{8}}  \leq 1.
\end{align}

\begin{figure}
    \centering
    \includegraphics[width=0.9\linewidth]{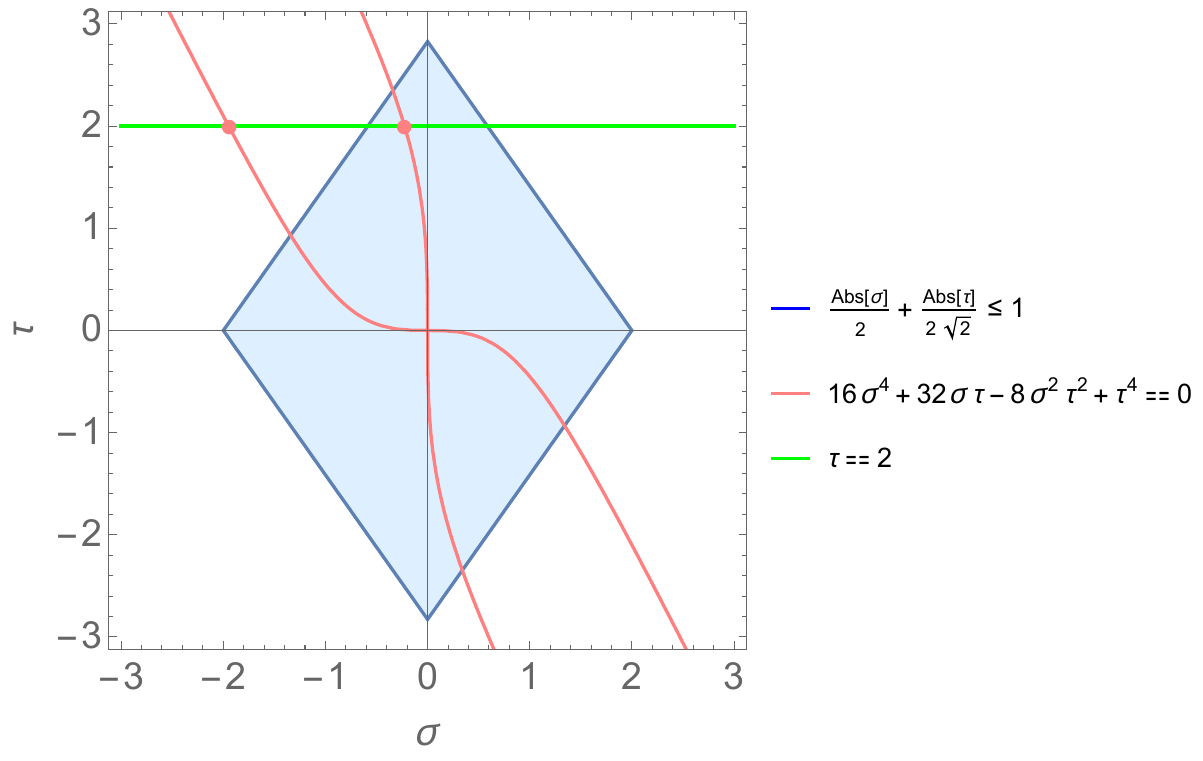}
    \vspace{-0.5cm}
    \caption{Visualization of Proposition \ref{step2} in Mathematica: the set $y(\sigma,\tau)$ (in blue), two branches of $h(\sigma,\tau)=0$ (in orange) and two intersection points $\omega_1, \omega_0$ (in orange) with the line $\tau=\tau_0$ (in green).}
    \label{fig:enter-label}
\end{figure}

\begin{proposition}\label{step2} Let $y(\sigma,\tau)$ be given by \eqref{yy}.  Then there are infinitely many solutions $(\sigma,\tau)$ to the equation \eqref{sigmatau} which also obey $y(\sigma,\tau)<1.$
\end{proposition}
\begin{proof}
Let
\begin{align*}
    h(\sigma,\tau)=\tau^4-8\tau^2\sigma^2+32\tau\sigma+16\sigma^4
\end{align*}
and notice that $h(-\frac{\tau}{2},\tau)=-16\tau^2$ and $h(0,\tau)=\tau^4$. Hence, for each $\tau \neq 0$ there is at least one $\sigma=\sigma(\tau)$ in the interval $(-\frac{\tau}{2},0)$ such that $h(\sigma(\tau),\tau)=0$.  Letting $\tau_0=2$ and solving $h(\sigma,\tau_0)=0$ (using, for example, Mathematica) produces two solutions, $\sigma_0$ and $\sigma_1$, say, which to 4 d.p. are
$$\sigma_0\simeq -0.2253 \ \ \mathrm{and} \ \ \sigma_1\simeq -1.9470.$$
More precisely, $\sigma_0$ and $\sigma_1$ are the only real roots of the polynomial $$p_4(\sigma)=1 + 4\sigma -2 \sigma^2+ \sigma^4.$$
Now, since $h_{\sigma}(\sigma,\tau_0)=64(\sigma^3-\sigma+1)$, it is easily checked that 
$$h_{\sigma}(\sigma_0,\tau_0) \neq 0,$$
and hence, by the Implicit Function Theorem, for suitably small $\epsilon>0$ there is a smooth branch 
\begin{align}\label{bee}
\mathcal{B}:=\{(\sigma(\tau),\tau): \tau \in (2-\eps,2+\eps)\}
\end{align}
of solutions to \eqref{sigmatau} emanating from the point $(\sigma_0,\tau_0)$.  Now we compute $$y(\sigma_0,\tau_0)\simeq 0.8197$$ to 4 d.p., and so it follows by continuity that 
$$y(\sigma(\tau),\tau) <1$$
for all $\tau$ sufficiently close to $\tau_0=2$.   By taking $\epsilon$ in the description of $\mathcal{B}$ smaller still if necessary, we can assume that $y(\sigma(\tau),\tau) < 1$ if $(\sigma(\tau),\tau) \in \mathcal{B}$. 
\end{proof}

Finally, using Lemma \ref{gap} and Proposition \ref{step2}, it follows that \eqref{tightrope2} holds and hence $F$ is mean coercive.  This enables us to prove, via Theorem \ref{minFgivesminD}, the following.

\begin{proposition}\label{culmination} Let $f$ be given by
\begin{align*}
    f = \sigma\chi_{Q_{2}}-(\tau+\sigma)\chi_{Q_{1}}+(\tau-\sigma)\chi_{Q_{3}}+\sigma \chi_{Q_{4}},
\end{align*}
where $(\sigma,\tau)$ belong to $\mathcal{B}$ as defined in \eqref{bee}.  Let $u$ be a solution of 
\begin{align*} \int_{Q} 2 \nabla u \cdot \nabla \varphi +  f(x)  \, \cof \nabla u \cdot \nabla \varphi \dx = 0 \quad \quad \varphi \in C_c^{\infty}(\om),
\end{align*}
 as provided by Proposition \ref{hydrostatic_solutions}, and
 define the function $g:=\det \nabla u$.  As before, set
 $$\sca_g=\{v\in H^1_{u_0}(Q;\R^2): \det \nabla v = g \ \textrm{a.e. in} \ Q\}$$
 where $u_0:=u$.   Then $u$ is a global minimizer of the Dirichlet energy $\mathbb{D}(u)$ in $\mathcal{A}_g$.    
\end{proposition}

\section{Numerical experiments in the planar case}\label{numericalresults}
 The MATLAB code of \cite{BKV23} based on \cite{MoVa} was extended to treat the non-homogeneous Dirichlet boundary condition $u_0$. A minimizer $u 
\in H^1_{u_0}(\om;\R^2)$ of \eqref{Fintro} is approximated by the finite element method (FEM) with the lowest order (known as P1) basis functions defined on a regular triangulation of the domain $\Omega$.  It is calculated using the trust-region method from the MATLAB Optimization Toolbox. 
The weak form \eqref{MKweakform} is discretized as the system of linear equations 
\begin{align}\label{MKweakform}
 (2 K_1 + K_2) \, \Vec{u} = 0.
 \end{align}
Here, a vector $\Vec{u} \in \mathbb{R}^{2n}$ represents the minimizer $u = (u_1, u_2)$ and $n$ denotes the number of triangulation nodes. 
The stiffness matrices $K_1$, $K_2 \in \mathbb{R}^{2n \times 2n}$ are constructed efficiently using the modification of \cite{RaVa} and correspond to the assembly of bilinear forms
\begin{align}
   & \int_{\om} \nabla\psi:\nabla u \dx= 
   \int_{\om}
    \begin{pmatrix}
        \frac{\partial \psi_1}{\partial x_1} & \frac{\partial \psi_1}{\partial x_2} \\
         \frac{\partial \psi_2}{\partial x_1} &  \frac{\partial \psi_2}{\partial x_2}
    \end{pmatrix}
    :
     \begin{pmatrix}
        \frac{\partial u_1}{\partial x_1} & \frac{\partial u_1}{\partial x_2} \\
         \frac{\partial u_2}{\partial x_1} &  \frac{\partial u_2}{\partial x_2}
    \end{pmatrix}
    \dx
 ,  
 \\
    & \int_{\om} f \, \cof\nabla\psi:\nabla u \dx= 
    \int_{\om} f \begin{pmatrix*}[r]
        \frac{\partial \psi_2}{\partial x_2} & -\frac{\partial \psi_2}{\partial x_1} \\
         -\frac{\partial \psi_1}{\partial x_2} &  \frac{\partial \psi_1}{\partial x_1}
    \end{pmatrix*}
    :
     \begin{pmatrix*}
        \frac{\partial u_1}{\partial x_1} & \frac{\partial u_1}{\partial x_2} \\
         \frac{\partial u_2}{\partial x_1} &  \frac{\partial u_2}{\partial x_2}
    \end{pmatrix*} \dx . 
\end{align}
The matrix $K_1$ is symmetric and is constructed of two identical stiffness matrices corresponding to the discretization of the Laplace operator for the scalar variable. 
The matrix $K_2$ is non-symmetric and combines the products of the mixed derivatives of the basis functions further weighted by the function $f$. The function $f$ is assumed to be a piecewise constant in smaller subdomains. If the triangulation is aligned with subdomain shapes, then the numerical quadrature of both terms in \eqref{Fintro} is exact. An additional mesh adaptivity is applied using the MATLAB Partial Differential Equation Toolbox to enhance accuracy across nonlinear subdomain boundaries; see Figures \ref{fig:mesh_disk_disk_adaptive}, \ref{fig:mesh_disk_section_adaptive}.
A complementary code is available for download and testing at
\begin{center}
\url{https://www.mathworks.com/matlabcentral/fileexchange/130564} .
\end{center}
\vspace{0.5mm}

 \begin{figure}
\includegraphics[width=0.8\textwidth]{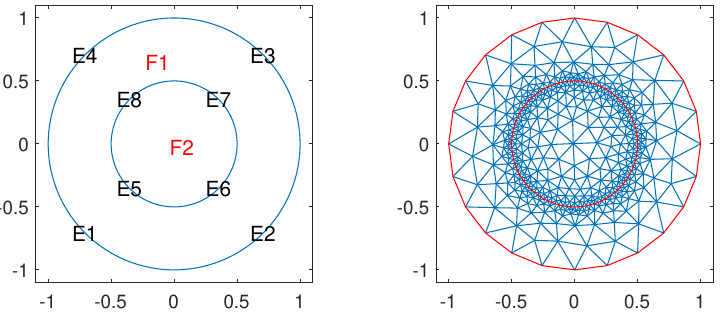}
\caption{A disk-disk geometry 
(left) and the example of its adaptive mesh refinement (right).
}
\label{fig:mesh_disk_disk_adaptive}
\vspace{1cm}
\includegraphics[width=0.8\textwidth]{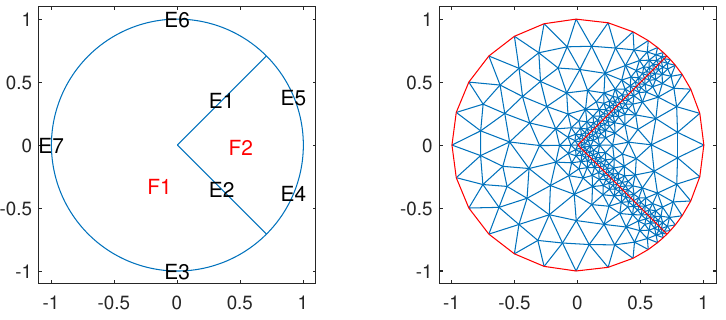}
\caption{A disk-section geometry 
(left) and the example of its adaptive mesh refinement (right).
}
\label{fig:mesh_disk_section_adaptive}
\vspace{1cm}
\includegraphics[width=\textwidth]{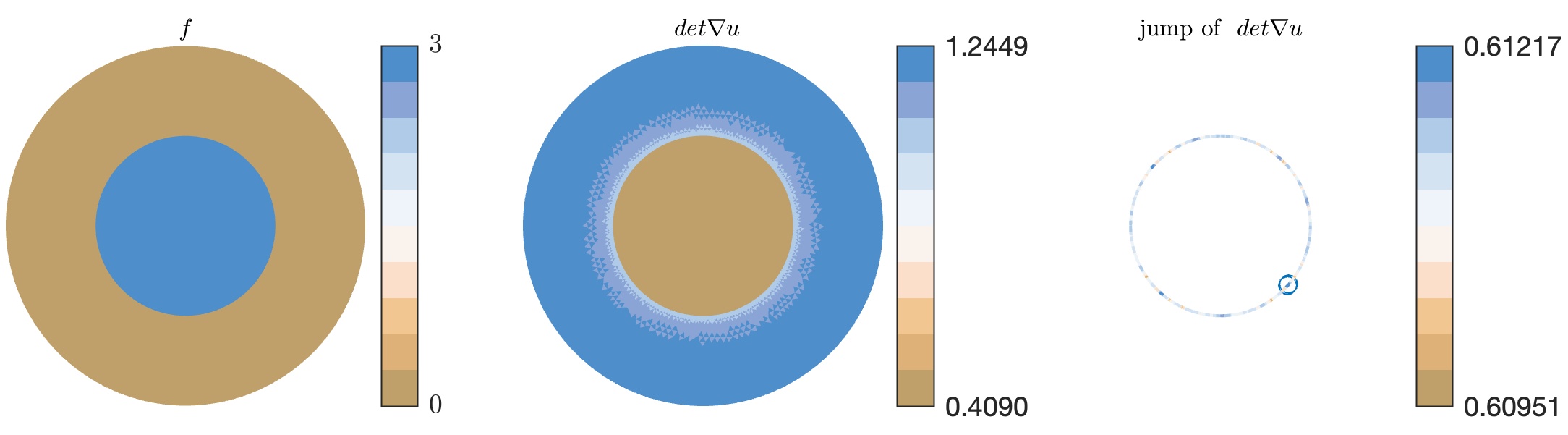}
\caption{Distribution of $f$ (left), $\det \nabla u$ (middle), and the jump 
of $\det \nabla u$ (right) across the interface boundary.  Here, the boundary condition in force is $u_0(x)=x$.
}
\label{fig:disk_disk_adaptive}
\end{figure}

 \subsection{Disk-disk problem}\label{diskdisk}
 Let us compare the analytical solution $u$, given by \eqref{diskdisksol}, to the Euler-Lagrange equation \eqref{ELinhombc}  with its numerically generated counterpart.  For concreteness we set the parameters $\rho=0.5$ (inner disk radius) and $M=3$, from which it follows that 
 $\zeta
 = 0.64, \xi
 = 1.12    $
and $\det \nabla u(x)$ is an axisymmetric function satisfying 
\begin{align}
 & \min_{x \in \Omega} \det \nabla u =  0.4096, \qquad \max_{x \in \Omega} \det \nabla u =  1.24   \\
 & \lim_{|x| \rightarrow \rho -} \det \nabla u - \lim_{|x| \rightarrow \rho +} \det \nabla u = 1.024 - 0.4096 = 0.6144.
\end{align}
The FEM calculation using 13930 triangles and 7066 nodes shows similar values: see Figure \ref{fig:disk_disk_adaptive} and, particularly, its colorbar limits.

 \subsection{Disk-sector problem}
Let us compare the analytical solution $u$ given by \eqref{example1}, \eqref{beeth5} to the Euler-Lagrange equation \eqref{ELinhombc} with those generated using the numerical methods described above.  The geometry is as shown in Figure \ref{fig:disk_sector_nid}, and the free parameter $M$ featuring in Subsection  \ref{disk-sector} is set equal to $3$.
 We find that $\det \nabla u(x)$ satisfies  
\begin{align}
 & \min_{x \in \Omega} \det \nabla u =  0, \qquad \max_{x \in \Omega} \det \nabla u =  6  \\
 & \max_{R \in [0,1]}  \left( 
 \lim_{|\theta| \rightarrow \pi /4 +} \det \nabla u 
 - 
 \lim_{|\theta| \rightarrow \pi /4 -} \det \nabla u
 \right) = 6.
\end{align}
The FEM calculation using 11316 triangles and 5765 nodes shows very 
similar values:  see Figure \ref{fig:disk_sector_nid} and, particularly, its colorbar limits.
\begin{figure}
\includegraphics[width=\textwidth]{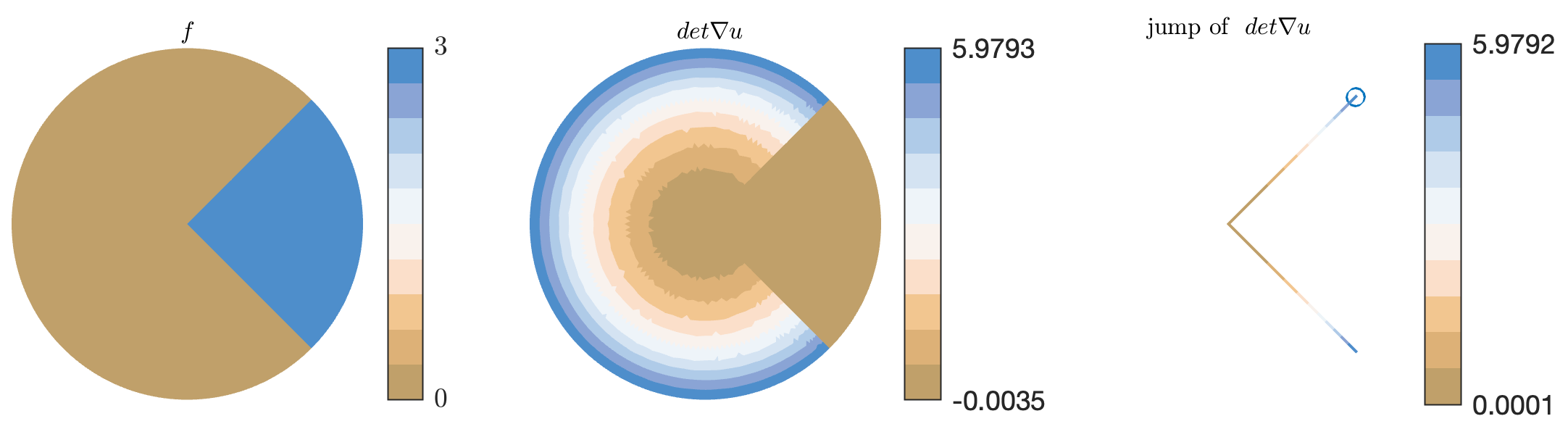}
\caption{Distribution of $f$ (left), $\det \nabla u$ (middle), and the jump (in modulus) of $\det \nabla u$ (right) across the interface boundary. The boundary condition $u_0$ is given by \eqref{beeth5} with $R:=1$;  in particular $u_0$ is not the identity map. 
}
\label{fig:disk_sector_nid}
\vspace{0.5cm}
\includegraphics[width=\textwidth]{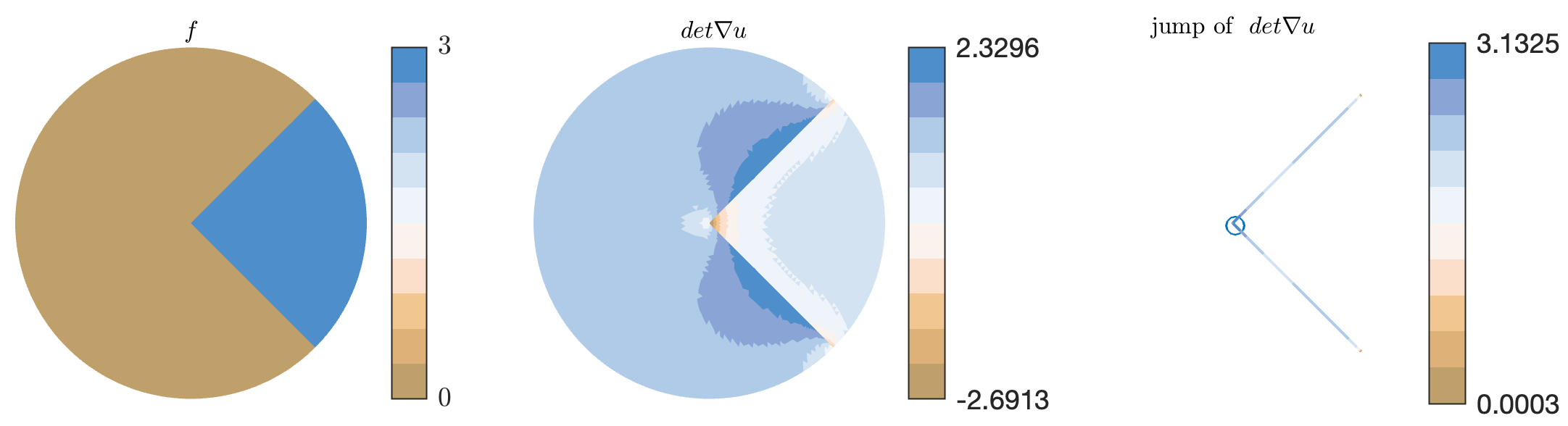}
\caption{Distribution of $f$ (left), $\det \nabla u$ (middle), and the jump (in modulus) of $\det \nabla u$ (right) across the interface boundary. 
Here, the boundary condition in force is $u_0(x)=x$.
}
\label{fig:disk_section_id}
\end{figure}

Our final numerical result goes beyond what we can say analytically.  Specifically, in Figure \ref{fig:disk_section_id}, aspects of the numerical solution $u$ to \eqref{ELinhombc} are shown when the boundary condition $u_0$ obeys $u_0(x)=x$ for $x \in \partial B$.  We cannot make a direct comparison with an analytical solution here because $u_0$ is not the suitably prepared type needed, for example, in Proposition \ref{tango}, and it is not clear how to render it so.

\bigskip 

\noindent
{\bf Acknowledgment.}

   MK and JV were partially supported by the GA\v{C}R project 23-04766S. They thank the Department of Mathematics of the University of Surrey for the hospitality during their stays there.  JB would like to thank MK, JV and UTIA, Czech Academy of Sciences for their hospitality during his visits. 


\begin{thebibliography}{99}
\bibitem{BKV23} J. Bevan, M. Kru\v{z}\'{i}k and J. Valdman. Hadamard's inequality in the mean.  Nonlinear Analysis {\bf(243)}, 
2024. https://doi.org/10.1016/j.na.2024.113523
\bibitem{Chkara09} N. Chaudhuri and A. Karakhanyan.  On derivation of Euler-Lagrange equations for incompressible energy-minimizers.   {\it{Calc. Var.}}  {\bf{36}} 627-645, 2009.

\bibitem{CKKK} J. Campos Cordero, B. Kirchheim, J. Kol\'{a}\v{r},  J. Kristensen. Convexity and uniqueness in the calculus of variations. In preparation.

\bibitem{Dac08} B. Dacorogna. {\it Direct Methods in the Calculus of Variations.} Applied Mathematical Sciences (2nd ed.), Springer, Berlin, 2008. 
\bibitem{EG92}  L. C. Evans and R. F. Gariepy.   Measure theory and fine properties of functions. In: Studies in Advanced Mathematics, CRC Press, Boca Raton (1992).  
\bibitem{FM86} R. Fosdick, G. P. MacSithigh. Minimization in incompressible nonlinear elasticity theory.  {\it{J. Elasticity}}. {\bf{16:}} 267-301, 1986.
\bibitem{FRC99} R. Fosdick and  G. Royer-Carfagni.  The Lagrange Multiplier in Incompressible Elasticity Theory. {\it{J. Elasticity}}. {\bf{55:}} 193-200, 1999.
\bibitem{Giaquinta} M. Giaquinta.  Multiple integrals in the calculus of variations and nonlinear elliptic systems. Annals of Mathematics Studies, 105. Princeton University Press, Princeton, NJ, 1983. 
\bibitem{GiaquintaGiusti1982} M. Giaquinta and E. Giusti.  On the regularity of minima of variational integrals. Acta Math. 148 (1982), 31-46.
\bibitem{Kara12} A. Karakhanyan.  Sufficient conditions for the regularity of area-preserving deformations.   {\it{Manuscripta Math.}} 138.  463-476, 2012.
\bibitem{Kara14} A. Karakhanyan.  Regularity for Energy-minimizing area-preserving deformations.  {\it{J. Elasticity}}. {\bf{114}} 213-223, 2014.
\bibitem{LO81} P. LeTallec and J. T.  Oden.  Existence and characterization of hydrostatic pressure in finite deformations of incompressible elastic bodies. {\it{J. Elasticity}}. {\bf{11 (4)}} 341-357, 1981.
\bibitem{Mo66}  C. B. Morrey, Jr.  Multiple integrals in the calculus of variations. Reprint of the 1966 edition. Classics in Mathematics. Springer-Verlag, Berlin, 2008. 
\bibitem{MoVa}
A. Moskovka, J. Valdman. Fast MATLAB evaluation of nonlinear energies using FEM in 2D and 3D: nodal elements. {\it Applied Mathematics and Computation} {\bf 424},  (2022) 127048.
\bibitem{Neffetal} D. Yang Gao, P. Neff, I. Roventa and C. Thiel.  On the Convexity of Nonlinear Elastic Energies in the Right Cauchy-Green Tensor. J Elast 127, 303–308 (2017). 
\bibitem{RaVa} T. Rahman, J. Valdman. Fast MATLAB assembly of FEM  matrices in 2D and 3D: nodal elements, {\it Applied Mathematics and Computation}, {\bf 219} , No. 13, (2013), 7151--7158. 

\end{thebibliography}
\end{document}